\documentclass{siamltex}
\usepackage{amsmath, amssymb, graphicx, hyperref, geometry}
\usepackage{listings}
\usepackage{xcolor}
\usepackage{mathtools}
\usepackage{latexsym}
\usepackage{subfig}
\usepackage{graphicx}
\usepackage[mathscr]{euscript}
\usepackage{amssymb}
\usepackage{fancyhdr}
\usepackage{multirow}
\usepackage{epsfig}
\usepackage{mathptmx}
\usepackage{bm}
\usepackage{pgfplots}
\pgfplotsset{compat=1.18}

\usepackage{algorithm}
\usepackage{algpseudocode}
\usepackage{amsmath}

\newtheorem{remark}{Remark}
\newcommand{\nd}{\noindent}

\title{A Unified Trace-Optimization Framework \\ for Multidimensionality Reduction}
\author{M. El Guide \thanks{FGSES, University Mohammed VI Polytechnic, Rabat, Morocco. E-mail: mohamed.elguide@um6p.ma} 
	\and A. Elichi\thanks{Université du Littoral Cote d'Opale, LMPA, 50 rue F. Buisson, 62228 Calais-Cedex, France. E-mail:elichi.alaa@gmail.com}
    \and K. Jbilou\thanks{Université du Littoral Cote d'Opale, LMPA, 50 rue F. Buisson, 62228 Calais-Cedex, France. E-mail: khalide.jbilou@univ-littoral.fr}
    \and L. Reichel\thanks{Department of Mathematics, Kent State University, Kent, OH 44122, USA. E-mail: reichel@math.kent.edu} \and H. Alqahtani\thanks{Department of Mathematics, King Abdulaziz University, Campus Rabigh 21911, Saudi Arabia. E-mail: hfalqahtani@kau.edu.sa}}

\begin{document}
	\maketitle
	
\begin{abstract}
This paper presents a comprehensive overview of several multidimensional reduction methods focusing on  Multidimensional Principal Component Analysis (MPCA), Multilinear Orthogonal Neighborhood Preserving Projection (MONPP), Multidimensional Locally Linear Embedding (MLLE), and  Multidimensional Laplacian Eigenmaps (MLE). These techniques are formulated within a unified framework based on trace optimization, where the dimensionality reduction problem is expressed as maximization or minimization problems. In addition to the linear MPCA and MONPP approaches, kernel-based extensions of these methods also are presented.
The latter methods make it possible to capture nonlinear relations between high-dimensional data. A comparative analysis highlights the theoretical foundations, assumptions, and computational efficiency of each method, as well as their practical applicability. The study provides insights and guidelines for selecting an appropriate dimensionality reduction technique suited to the application at hand.
\end{abstract}

\section{Introduction}
High-dimensional data arise in many modern applications including computer vision, pattern recognition, hyperspectral imaging, audio analysis, and data visualization. In these applications each observation may be represented by a large number of variables, often with strong correlations and complex geometric structure. Direct learning is typically hampered by the curse of dimensionality, computational cost, and interpretability. Dimensionality reduction techniques address these issues by seeking low-dimensional representations that preserve the most relevant information in the data, for example in terms of variance, neighborhood structure, or class separability.

A classical and widely used approach is Principal Component Analysis (PCA) \cite{Hotelling1933,Jolliffe2002}. PCA is a linear and unsupervised method that finds orthogonal directions, known as principal components, that capture the largest variance of the data. By retaining only a small number of leading components, PCA provides a compact representation that often is sufficient for tasks such as visualization, compression, and noise reduction. However, since PCA does not exploit label information, its discriminative power may be limited in supervised tasks such as classification.

To overcome this limitation, many supervised linear dimensionality reduction methods have been proposed. Linear Discriminant Analysis (LDA) \cite{Dufrenois2023,Fisher1936,Rao1948} and related Fisher-type criteria seek to determine projections that maximize the ratio between between-class and within-class scatter, thereby enhancing class separability in the reduced space. Moreover, a rich family of manifold and graph-based techniques has been developed to preserve local geometric relations between samples. Examples include Locally Linear Embedding (LLE) \cite{Roweis2000}, Laplacian Eigenmaps (LE) \cite{belkin2003}, Locality Preserving Projections (LPP) \cite{He2003}, and Orthogonal Neighborhood Preserving Projections (ONPP) \cite{Kokiopoulou2007}. Kernel extensions of many of these methods, together with trace optimization formulations and generalized eigenvalue problems, have been investigated in, e.g.\cite{Kokiopoulou2011,Muller2001,Ngo2012,Scholkopf1998,Shawe2004}.

Most of the above mentioned methods operate on vectorial data and therefore assume that each sample can be represented as a vector in an Euclidean space. In many applications, however, data are naturally multidimensional. Typical examples include color and multispectral images, video sequences, tensor-valued time series, and multisensor measurements. A common workaround consists in vectorizing the original multidimensional arrays into long vectors before applying matrix-based techniques such as PCA or LDA. However, this strategy often destroys important spatial or multimode correlations, leads to very high ambient dimensions, and may degrade both statistical and computational efficiency. To address these issues, multilinear and tensor-based dimensionality reduction techniques have been proposed, for example through higher-order singular value and tensor decompositions; see  \cite{deLathauwer2000,Lu2008MPCA,Vasilescu2002}, as well as graph- and manifold-based tensor methods.

In recent years, the t-product framework has emerged as a powerful algebraic tool for third-order tensors; see  \cite{Avron2025,ElHachimi2023,ElHachimi2024,Kilmer2013}. This formalism makes it possible to extend many familiar matrix concepts, including the singular value and spectral decompositions, to tensors while preserving a convenient linear-algebra structure via block circulant representations and Fourier transforms. Several numerical linear algebra methods and spectral techniques have been successfully generalized to this setting \cite{Elguide2021,ElHachimi2023,ElHachimi2024}. This suggests that it may be natural to revisit classical dimensionality reduction criteria using the t-product framework.

In this work, we develop a unified trace-based optimization framework for multidimensional dimensionality reduction in the t-product setting. Building on tensor trace operators and generalized eigenvalue problems, we extend several well-known matrix-based methods for dimensionality reduction to third-order tensors. More precisely, we introduce and analyze Multidimensional Principal Component Analysis (MPCA) and Multidimensional Orthogonal Neighborhood Preserving Projections (MONPP), both formulated in terms of tensor quantities. We also derive kernel extensions of these methods that enable modeling of nonlinear structures in multidimensional data, while retaining the computational advantages of the t-product algebra. In addition, we consider multidimensional variants of nonlinear manifold learning techniques such as Laplacian Eigenmaps (MLE) and Locally Linear Embedding (MLLE), and show how these techniques can be expressed within the framework of trace-optimization techniques.

The proposed framework provides a coherent view of linear, kernel, and nonlinear multidimensional reduction methods based on trace criteria and tensor algebra. It clarifies the links between different approaches, identifies the corresponding tensor eigenvalue or generalized eigenvalue problems, and leads to practical algorithms that operate directly on tensor data without vectorization.

The remainder of this paper is organized as follows. Section~\ref{sec:preliminaries} recalls properties of the t-product that are important for our analysis. In Section~\ref{sec:tensor-trace}, we develop some theoretical results on tensor trace optimization. These results form the basis for the proposed multidimensional reduction methods. Section~\ref{sec:MPCA} introduces the Multidimensional Principal Component Analysis (MPCA) method and Section~\ref{sec:MONPP} is discusses the Multidimensional Orthogonal Neighborhood Preserving Projection (MONPP) method. Kernel extensions of these methods are presented in Section~\ref{sec:MKBM}. Section~\ref{sec:MNRM} discusses multidimensional nonlinear methods, including Multidimensional Laplacian Eigenmaps (MLE) and Multidimensional Locally Linear Embedding (MLLE). Section~\ref{sec:NE} reports numerical experiments for several benchmark datasets. These examples illustrate the effectiveness and computational behavior of the proposed approaches. Finally, Section \ref{sec:concl} contains concluding remarks. 

\section{Preliminaries}\label{sec:preliminaries}
This paper is concerned with third-order tensors 
$\mathcal{A} \in \mathbb{R}^{m \times n \times p}$. 
Tensors are denoted by boldface capital calligraphic letters such as 
$\mathcal{A}$. Standard capital letters, such as $A$, are used for matrices; boldface lower case letters, such as $\boldsymbol{a}$, denote tubes; and standard lower case letters, such as $a$, denote vectors. The $(i,j,k)$th entry of the third-order tensor 
$\mathcal{A} \in \mathbb{R}^{m \times n \times p}$ is denoted by 
$\mathcal{A}_{i,j,k}$.

A slice of a third-order tensor ${\mathcal A}$ is obtained by fixing one index. We let $\mathcal{A}(i,:,:) \in \mathbb{R}^{1 \times n \times p}$ denote the $i$th horizontal slice, 
$\mathcal{A}(:,j,:) \in \mathbb{R}^{m \times 1 \times p}$ stands for the $j$th lateral slice, and 
$\mathcal{A}(:,:,k) \in \mathbb{R}^{m \times n \times 1}$ the $k$th frontal slice. Fixing two indices yields a fiber of $\mathcal{A}$. 
A column (1-mode), row (2-mode), and tube (3-mode) fiber are denoted 
by $\boldsymbol{a}_{:jk}$, $\boldsymbol{a}_{i:k}$, and 
$\boldsymbol{a}_{ij:}$, respectively.

The Frobenius inner product of two third-order tensors $\mathcal{A},\mathcal{B} \in \mathbb{R}^{m \times n \times p}$ is defined as
\[
\langle \mathcal{A},\mathcal{B} \rangle
  = \sum_{i=1}^m \sum_{j=1}^n \sum_{k=1}^p 
      \mathcal{A}_{i,j,k} \mathcal{B}_{i,j,k},
\]
and the Frobenius norm of the third-order tensor $\mathcal{A}$ is given by
\[
\|\mathcal{A}\|_F
  = \|\mathcal{A}\|_F=\sqrt{\sum_{i=1}^m \sum_{j=1}^n \sum_{k=1}^p 
           \mathcal{A}_{i,j,k}^2}.
\]

The Fast Fourier Transform (FFT) will be used frequently in this paper. Let $F_\ell \in \mathbb{C}^{\ell \times \ell}$ denote the Discrete 
Fourier Transform (DFT) matrix of order $\ell$. Its elements are given by
\[
(F_\ell)_{k,j} = \frac{1}{\sqrt{\ell}} \,\omega^{(k-1)(j-1)},
\quad k,j = 1,\ldots,\ell,
\]
where $\omega = e^{-2\pi i / \ell}$ and $i=\sqrt{-1}$. The matrix $F_\ell$ is unitary; it satisfies $F_\ell^H F_\ell = I_\ell$,
where the superscript $^H$ denotes transposition and complex conjugation.

The DFT of $v \in \mathbb{C}^\ell$ is given by $\widehat{v} = F_\ell v$. 
If $v \in \mathbb{R}^\ell$, then $\widehat{v}$ satisfies the conjugate symmetry relations
\begin{equation}\label{even conj}
  \widehat{v}_1 \in \mathbb{R}, \qquad 
  {\tt conj}(\widehat{v}_i) = \widehat{v}_{\ell - i + 2},
  \quad i = 2,\ldots,
  \left\lfloor\frac{\ell+1}{2}\right\rfloor,
\end{equation}
where $\lfloor\alpha\rfloor$ denotes the largest integer smaller than or equal to $\alpha\ge 0$.

A vector $\widehat{v}\in{\mathbb C}^\ell$ that satisfies 
\eqref{even conj} is said to be conjugate even. Moreover, if a vector 
$\tilde{v}$ is real and conjugate even, then its inverse DFT 
$v = F_\ell^H \tilde{v}$ also is real and conjugate even. 

Let $\mathcal{A} \in \mathbb{R}^{m \times n \times p}$ be a third-order tensor with frontal slices $A_1, A_2,\ldots,A_p$. The operator 
${\tt Bcirc}$ constructs a block circulant matrix from the frontal slices of $\mathcal{A}$:
\[
{\tt Bcirc}(\mathcal{A})
  = \begin{pmatrix}
      A_1   & A_p   & \cdots & A_2 \\
      A_2   & A_1   & \cdots & A_3 \\
      \vdots& \ddots& \ddots & \vdots \\
      A_p   & A_{p-1} & \cdots & A_1
    \end{pmatrix}
    \in \mathbb{R}^{mp \times np}.
\]
The operator ${\tt MatVec}$ and its inverse ${\tt MatVec}^{-1}$ are defined as
\[
{\tt MatVec}(\mathcal{A})
  = \begin{pmatrix}
      A_1 \\
      \vdots \\
      A_p
    \end{pmatrix}
    \in \mathbb{R}^{mp \times n},
\qquad
{\tt MatVec}^{-1}({\tt MatVec}(\mathcal{A})) = \mathcal{A}.
\]
The block diagonal matrix associated with $\mathcal{A}$ is defined as
\begin{equation}\label{dft9}
  {\tt Bdiag}(\mathcal{A})
    = \operatorname{BlocDiag}(A_1,\ldots,A_p),
\end{equation}
where $A_j$, $j=1,\ldots,p$, are the frontal slices of $\mathcal{A}$.

Let $\widehat{\mathcal{A}}$ denote the tensor obtained by applying the DFT along 
the third mode (along all tubes) of $\mathcal{A}$. In Matlab notation,
\[
\widehat{\mathcal{A}} = {\tt fft}(\mathcal{A},[],3),
\qquad
{\tt ifft}(\widehat{\mathcal{A}},[],3) = \mathcal{A},
\]
where ${\tt fft}$ and ${\tt ifft}$ denote the FFT and inverse FFT functions. The block circulant matrix ${\tt Bcirc}(\mathcal{A})$ associated with the third-order tensor 
$\mathcal{A} \in \mathbb{R}^{m \times n \times p}$ is block diagonalizable using the DFT:
\begin{equation}\label{dft8}
  (F_p \otimes I_m)\,{\tt Bcirc}(\mathcal{A})\,
  (F_p^H \otimes I_n)
  = {\tt Bdiag}(\widehat{\mathcal{A}}).
\end{equation}
If $\mathcal{A}$ is real, then the frontal slices of $\widehat{\mathcal{A}}$ satisfy
\begin{equation}\label{f1}
  \widehat{\mathcal{A}}_1 \in \mathbb{R}^{m \times n}, 
  \qquad
  {\tt conj}(\widehat{\mathcal{A}}_i)
    = \widehat{\mathcal{A}}_{p - i + 2},
  \quad i = 2,\ldots,
  \left\lfloor\frac{p+1}{2}\right\rfloor,
\end{equation}
where ${\tt conj}(\widehat{\mathcal{A}}_i)$ senotes the complex conjugate of the matrix $\widehat{\mathcal{A}}_i$. Kilmer and Martin \cite{Kilmer2011} introduced the following tensor product.

\begin{definition}
Let $\mathcal{A} \in \mathbb{R}^{m \times \ell \times p}$ and 
$\mathcal{B} \in \mathbb{R}^{\ell \times n \times p}$. 
The t-product of $\mathcal{A}$ and $\mathcal{B}$ is the tensor 
$\mathcal{C} \in \mathbb{R}^{m \times n \times p}$ defined by
\[
\mathcal{C}
  = \mathcal{A} * \mathcal{B}
  = {\tt MatVec}^{-1}
      \bigl[\,{\tt Bcirc}(\mathcal{A})\,
             {\tt MatVec}(\mathcal{B})\,\bigr].
\]
\end{definition}

\nd Let $\mathcal{C} = \mathcal{A} * \mathcal{B}$ be the t-product of 
$\mathcal{A} \in \mathbb{R}^{m \times \ell \times p}$ and 
$\mathcal{B} \in \mathbb{R}^{\ell \times n \times p}$. 
Then the FFT of $\mathcal{C}$ can be computed slice-wise as
\[
\widehat{\mathcal{C}} = \widehat{\mathcal{A}} \triangle \widehat{\mathcal{B}},
\]
where $\triangle$ denotes the face-wise (frontal slice-wise) product, i.e.,
\[
\widehat{\mathcal{C}}_i
  = \widehat{\mathcal{A}}_i \widehat{\mathcal{B}}_i,
  \quad i = 1,\ldots,p.
\]

We review some definitions and properties of tensor operations of third-order tensors when using the t-product. They have been shown by Kilmer et al. \cite{Kilmer2011,Kilmer2013} and will be used throughout this paper. 

\begin{definition}
(i) Let $\mathcal{A} \in \mathbb{R}^{m \times n \times p}$.
The t-transpose of $\mathcal{A}$ is the tensor 
$\mathcal{A}^T \in \mathbb{R}^{n \times m \times p}$ obtained by transposing each 
frontal slice and then reversing the order of the transposed frontal slices $A_2^T,\ldots,A_p^T$. (ii) The identity tensor $\mathcal{I}_{mmp}$ is the tensor of size $m \times m \times p$ whose first frontal slice is the $m \times m$ identity matrix and whose other frontal slices are zero matrices. We note that in the Fourier domain all frontal slices of 
$\widehat{\mathcal{I}_{mmp}}$ equal the identity matrix $I_m$. 
(iii) A square tensor $\mathcal{Q} \in \mathbb{R}^{m \times m \times p}$ is said to be orthogonal if
\[
\mathcal{Q}^T * \mathcal{Q}
  = \mathcal{Q} * \mathcal{Q}^T
  = \mathcal{I}_{mmp}.
\]
(iv) The tensor $\mathcal{Q}$ is said to be f-orthogonal if all frontal slices 
$\widehat{Q}_i$, $i=1,\ldots,p$, of 
$\widehat{\mathcal{Q}} = {\tt fft}(\mathcal{Q},[],3)$ are orthogonal matrices. 
(v) A tensor $\mathcal{D} \in \mathbb{R}^{m \times m \times p}$ is said to be f-diagonal if all its frontal slices in the Fourier domain are diagonal matrices.
(vi) A square tensor $\mathcal{A} \in \mathbb{R}^{m \times m \times p}$ is invertible if 
there is a tensor $\mathcal{A}^{-1} \in \mathbb{R}^{m \times m \times p}$ such that
\[
\mathcal{A} * \mathcal{A}^{-1}
  = \mathcal{I}_{mmp}.
\]
\end{definition}

\begin{definition}\label{def2.3}
Let $\mathcal{A} \in \mathbb{R}^{m \times m \times p}$. 
The tensor $\mathcal{A}$ is said to be f-symmetric if all its frontal slices $\widehat{A}_i$, $i=1,\ldots,p$, in the Fourier domain are symmetric matrices. The tensor $\mathcal A$ is said to be 
f-positive-definite (f-positive-semidefinite) if all its frontal slices in the Fourier domain are symmetric positive definite (positive-semidefinite) matrices. 
\end{definition}

The singular value decomposition (SVD) of a matrix can be extended to a third-order tensor in the t-product setting; see, for example, 
\cite{ElHachimi2024,Kilmer2011,Kilmer2013}.

\begin{theorem}\label{theosvd1}
Let $\mathcal{A} \in \mathbb{R}^{m \times n \times p}$ be a real-valued tensor. 
Then there exist tensors 
$\mathcal{U} \in \mathbb{R}^{m \times m \times p}$, 
$\mathcal{S} \in \mathbb{R}^{m \times n \times p}$, and 
$\mathcal{V} \in \mathbb{R}^{n \times n \times p}$ such that
\begin{equation}\label{svd1}
  \mathcal{A}
    = \mathcal{U} * \mathcal{S} * \mathcal{V}^T,
\end{equation}
where $\mathcal{U}$ and $\mathcal{V}$ are f-orthogonal tensors and 
$\mathcal{S}$ is an f-diagonal tensor.
\end{theorem}

\nd A generalization of eigenvalues and eigenvectors to third-order tensors using the 
t-product has been described in  \cite{ElHachimi2023}.

\begin{definition}
Let $\mathcal{A} \in \mathbb{R}^{m \times m \times p}$ be a third-order tensor. 
A tube ${\bm{\lambda}}_j \in \mathbb{C}^{1 \times 1 \times p}$ is called an 
eigentube of $\mathcal{A}$ corresponding to the eigen-lateral slice 
$\vec{\mathcal{V}}_j \in \mathbb{C}^{m \times 1 \times p}$ if
\begin{equation}\label{teig1}
  \mathcal{A} * \vec{\mathcal{V}}_j
    = \vec{\mathcal{V}}_j * {\bm{\lambda}}_j.
\end{equation}
For each $j = 1,\ldots,m$, let 
$\lambda_j(\widehat{A}_k)$, $k=1,\ldots,p$, denote the $j$th eigenvalue of the $k$-th frontal slice $\widehat{A}_k$ of $\widehat{\mathcal{A}}$. The $j$th eigentube ${\bm{\lambda}}_j$ is defined as
\[
{\bm{\lambda}}_j(1,1,k) = \lambda_j(\widehat{A}_k),
\quad k = 1,\ldots,p,
\]
and the eigenvalues in each frontal slice are ordered so that
\[
|\lambda_\ell(\widehat{A}_k)|
  \ge |\lambda_{\ell+1}(\widehat{A}_k)|,
\quad \ell = 1,\ldots,m-1,
\quad k = 1,\ldots,p.
\]
\end{definition}

\begin{definition}[Tensor f-diagonalization]\label{def16}
A third-order tensor $\mathcal{A} \in \mathbb{R}^{m \times m \times p}$ is said to be  f-diagonalizable if it is similar, via the t-product, to an
f-diagonal tensor, that is, if
\[
\mathcal{A}
  = \mathcal{P} * \mathcal{D} * \mathcal{P}^{-1},
\]
for some invertible tensor 
$\mathcal{P} \in \mathbb{R}^{m \times m \times p}$ and an f-diagonal tensor 
$\mathcal{D} \in \mathbb{R}^{m \times m \times p}$. 
In this case, $\mathcal{P}$ and $\mathcal{D}$ contain the eigenslices and 
eigentubes of $\mathcal{A}$, respectively.
\end{definition}

If the third-order tensor $\mathcal{A}$ is f-symmetric and f-positive-definite, 
then it is f-diagonalizable and all its eigenvalues 
$\lambda_i(\widehat{A}_j)$, $i=1,\ldots,m$, $j=1,\ldots,p$, are real and 
positive \cite{ElHachimi2023}.

\section{Tensor trace optimization problems}\label{sec:tensor-trace}
We first introduce a notion of trace for third-order tensors. It is defined in the Fourier domain.

\medskip
\begin{definition}[f-tensor trace]\label{defTTrace}
Consider a square tensor 
$\mathcal{A} \in \mathbb{R}^{m \times m \times p}$ with frontal slices $A_1,\ldots,A_p$. Let $\widehat{\mathcal{A}} = {\tt fft}(\mathcal{A},[],3)$ and let $\widehat{A}_i$, $i=1,\ldots,p$, denote the frontal slices of $\widehat{\mathcal{A}}$. The frontal trace of $\mathcal{A}$ is defined as
\begin{equation}\label{tr1}
  \operatorname{Trace}_f(\mathcal{A})
    = \frac{1}{p} \sum_{i=1}^{p} \operatorname{Trace}(\widehat{A}_i), 
\end{equation}
where $\operatorname{Trace}({\widehat A}_i)$ denotes the trace of the square matrix ${\widehat A}_i$, which is the $i$th frontal slice of the tensor ${\widehat {\mathcal A}}$. 
\end{definition}

The next result follows directly from the properties of the t-product, the FFT, and the definition of $\operatorname{Trace}_f$.

\begin{proposition}\label{pro1}
Let $\mathcal{A},\mathcal{B} \in \mathbb{R}^{m \times m \times p}$. 
Then
\begin{equation}\label{nm2}
  \operatorname{Trace}_f(\mathcal{A}^T * \mathcal{A})
    = \frac{1}{p} \sum_{i=1}^p \|\widehat{A}_i\|_F^2
    = \|\mathcal{A}\|_F^2,
\end{equation}
\[
  \operatorname{Trace}_f(\mathcal{A}^T * \mathcal{B})
    = \frac{1}{p} \sum_{i=1}^{p} \langle \widehat{A}_i, \widehat{B}_i \rangle
    = \langle \mathcal{A},\mathcal{B}\rangle,
\]
where $\langle \cdot,\cdot\rangle$ denotes the Frobenius inner product.
\end{proposition}

\medskip
\nd Let ${\mathcal A}$ be an f-symmetric and f-positive-definite third-order tensor. We will consider optimization problems of the form
\[
\max_{\mathcal{V}} \operatorname{Trace}_f(\mathcal{V}^T * \mathcal{A} * \mathcal{V}),
\]
where the third-order tensor ${\mathcal V}$ satisfies certain 
orthogonality or generalized orthogonality constraints.
These types of problems arise naturally in multidimensional dimensionality reduction. Their solutions are expressed in terms of tensor eigenstructures associated with the t-product.

\medskip
\begin{theorem}\label{theotrace}
Let $\mathcal{A}\in \mathbb{R}^{m \times m \times p}$ be an f-symmetric and f-positive-definite tensor, and let $\mathcal{V}\in \mathbb{R}^{m \times d \times p}$. Then
\begin{equation}\label{maxtrace1}
  \max_{\mathcal{V}^T * \mathcal{V} = \mathcal{I}}
    \operatorname{Trace}_f\left( \mathcal{V}^T * \mathcal{A} * \mathcal{V} \right)
  = \frac{1}{p} \sum_{i=1}^p \sum_{j=1}^{d} \lambda_j(\widehat{A}_i)
  = \frac{1}{p} \sum_{j=1}^d \|{\bm{\lambda}}_j\|_1,
\end{equation}
where for each $i=1,\ldots,p$, $\lambda_1(\widehat{A}_i)\ge\cdots\ge \lambda_m(\widehat{A}_i)$ are the eigenvalues of the frontal slice $\widehat{A}_i$, 
$\lambda_j(\widehat{A}_i)$, $j=1,\ldots,d$, denote its $d$ largest eigenvalues, 
and the $j$th eigentube ${\bm{\lambda}}_j\in\mathbb{R}^{1\times 1\times p}$ is defined as
\begin{equation}\label{etube}
  {\bm{\lambda}}_j(1,1,i) = \lambda_j(\widehat{A}_i), 
  \quad i=1,\ldots,p,
\end{equation}
so that $\|{\bm{\lambda}}_j\|_1 = \sum_{i=1}^p \lambda_j(\widehat{A}_i)$.
\end{theorem}

\begin{proof}
By Definition~\ref{defTTrace} and by properties of the FFT and the t-product, we have
\[
\operatorname{Trace}_f\left( \mathcal{V}^T * \mathcal{A} * \mathcal{V} \right)
  = \frac{1}{p} \sum_{i=1}^p 
      \operatorname{Trace}\Bigl( 
        \bigl(\widehat{\mathcal{V}^T * \mathcal{A} * \mathcal{V}}\bigr)_i
      \Bigr)
  = \frac{1}{p} \sum_{i=1}^p 
      \operatorname{Trace}\left( 
        \widehat{V}_i^T \widehat{A}_i \widehat{V}_i
      \right),
\]
where $\widehat{\mathcal{V}} = {\tt fft}(\mathcal{V},[],3)$ and 
$\widehat{V}_i$ denotes its $i$-th frontal slice. The constraint 
$\mathcal{V}^T * \mathcal{V} = \mathcal{I}$ is equivalent, in the Fourier domain, to
\[
  \widehat{V}_i^T \widehat{V}_i = I_d,\quad i=1,\ldots,p.
\]
Therefore,
\[
\max_{\mathcal{V}^T * \mathcal{V} = \mathcal{I}}
  \operatorname{Trace}_f\left( \mathcal{V}^T * \mathcal{A} * \mathcal{V} \right)
= \frac{1}{p}\sum_{i=1}^p
   \max_{\widehat{V}_i^T \widehat{V}_i = I_d}
     \operatorname{Trace}\left( 
       \widehat{V}_i^T \widehat{A}_i \widehat{V}_i
     \right).
\]
Since each $\widehat{A}_i$ is symmetric and positive-definite, the classical matrix result applies:
\[
  \max_{\widehat{V}_i^T \widehat{V}_i = I_d}
    \operatorname{Trace}(\widehat{V}_i^T \widehat{A}_i \widehat{V}_i)
  = \sum_{j=1}^d \lambda_j(\widehat{A}_i),
\]
where $\lambda_j(\widehat{A}_i)$, $j=1,\ldots,d$, are the $d$ largest eigenvalues of $\widehat{A}_i$. Hence
\[
  \max_{\mathcal{V}^T * \mathcal{V} = \mathcal{I}}
    \operatorname{Trace}_f\left( \mathcal{V}^T * \mathcal{A} * \mathcal{V} \right)
  = \frac{1}{p} \sum_{i=1}^p \sum_{j=1}^d \lambda_j(\widehat{A}_i).
\]
Grouping terms by eigentubes yields
\[
  \frac{1}{p} \sum_{i=1}^p \sum_{j=1}^d \lambda_j(\widehat{A}_i)
  = \frac{1}{p} \sum_{j=1}^d \|{\bm{\lambda}}_j\|_1,
\]
with ${\bm{\lambda}}_j$ defined by \eqref{etube}.
An optimizer $\mathcal{V}$ can be constructed by choosing each $\widehat{V}_i$ to have as columns the eigenvectors of $\widehat{A}_i$ associated with the $d$ largest eigenvalues.
\end{proof}

\medskip
\nd We now extend this result to the case when the constraint involves an f-positive-definite tensor.

\medskip
\begin{theorem}\label{theotracegen}
Let $\mathcal{A},\mathcal{B} \in \mathbb{R}^{m \times m \times p}$ be f-symmetric and f-positive-definite tensors, and let $\mathcal{V}\in \mathbb{R}^{m \times d \times p}$. Then
\begin{equation}\label{maxtracegen1}
  \max_{\mathcal{V}^T * \mathcal{B} * \mathcal{V} = \mathcal{I}}
    \operatorname{Trace}_f \left( \mathcal{V}^T * \mathcal{A} * \mathcal{V} \right)
  = \frac{1}{p} \sum_{i=1}^p \sum_{j=1}^{d} \mu_{i,j},
\end{equation}
where for each $i=1,\ldots,p$, the
$\mu_{i,1}\ge \cdots \ge \mu_{i,m}$ are the generalized eigenvalues of the matrix pair $(\widehat{A}_i,\widehat{B}_i)$. In particular, $\mu_{i,j}$, $j=1,\ldots,d$, denote the $d$ largest generalized eigenvalues. Here 
$\widehat{\mathcal{A}}$ and $\widehat{\mathcal{B}}$ are obtained by applying the FFT along the third mode. Moreover, $(\widehat{A}_i,\widehat{B}_i)$ denote the $i$th pair of frontal slices. 
\end{theorem}

\begin{proof}
By Definition~\ref{defTTrace} and properties of the t-product, we have
\[
  \operatorname{Trace}_f(\mathcal{V}^T * \mathcal{A} * \mathcal{V})
    = \frac{1}{p}\sum_{i=1}^p 
        \operatorname{Trace}\left(
          (\widehat{\mathcal{V}^T * \mathcal{A} * \mathcal{V}})_i
        \right)
    = \frac{1}{p}\sum_{i=1}^p 
        \operatorname{Trace}\left(
          \widehat{V}_i^T \widehat{A}_i \widehat{V}_i
        \right),
\]
where $\widehat{\mathcal{V}} = {\tt fft}(\mathcal{V},[],3)$ and $\widehat{V}_i$ is the $i$th frontal slice of $\widehat{\mathcal{V}}$. The constraint 
$\mathcal{V}^T * \mathcal{B} * \mathcal{V} = \mathcal{I}$ is equivalent to
\[
  \widehat{V}_i^T \widehat{B}_i \widehat{V}_i = I_d,
  \quad i=1,\ldots,p.
\]
Therefore,
\[
\max_{\mathcal{V}^T * \mathcal{B} * \mathcal{V} = \mathcal{I}}
  \operatorname{Trace}_f(\mathcal{V}^T * \mathcal{A} * \mathcal{V})
= \frac{1}{p}\sum_{i=1}^p
   \max_{\widehat{V}_i^T \widehat{B}_i \widehat{V}_i = I_d}
     \operatorname{Trace}(\widehat{V}_i^T \widehat{A}_i \widehat{V}_i).
\]
Since the matrices $\widehat{A}_i$ and $\widehat{B}_i$ are symmetric and $\widehat{B}_i$ is positive-definite, we obtain for each $i$ the matrix optimization problem
\[
  \max_{\widehat{V}_i^T \widehat{B}_i \widehat{V}_i = I_d}
    \operatorname{Trace}(\widehat{V}_i^T \widehat{A}_i \widehat{V}_i).
\]
It is solved by letting $\widehat{V}_i$ be made up of the $d$ generalized eigenvectors associated with the $d$ largest generalized eigenvalues
$\mu_{i,1},\ldots,\mu_{i,d}$ of the pair $(\widehat{A}_i,\widehat{B}_i)$. The optimal value equals $\sum_{j=1}^d \mu_{i,j}$. Summing over $i$ yields \eqref{maxtracegen1}.
\end{proof}

\begin{remark}
Using Lagrange multipliers directly at the tensor level, the first-order optimality condition for the problem in Theorem~\ref{theotracegen} leads to the generalized eigentube equation
\[
  \mathcal{A} * \mathcal{V} = \mathcal{B} * \mathcal{V} * \mathcal{D},
\]
where $\mathcal{D}$ is an f-diagonal tensor whose tubes ${\bm{\lambda}}_j$, $j=1,\ldots,d$, are generalized eigentubes of the pair $(\mathcal{A},\mathcal{B})$. Equivalently,
\[
  \mathcal{A} * \vec{\mathcal{V}}_j
    = \mathcal{B} * \vec{\mathcal{V}}_j * {\bm{\lambda}}_j,
  \quad j=1,\ldots,d,
\]
where $\vec{\mathcal{V}}_j$ denotes the $j$th lateral slice of 
$\mathcal{V}$.
\end{remark}

Analogous results hold for related minimization problems. We first consider the minimization problem that is obtained by replacing maximization by minimization in \eqref{theotracegen}. The solution of  problem will be used below.

\begin{theorem}\label{theotracemin}
Let $\mathcal{A}\in \mathbb{R}^{m \times m \times p}$ be an f-symmetric and f-positive-definite tensor, and let $\mathcal{V}\in \mathbb{R}^{m \times d \times p}$. Then
\begin{equation}\label{mintrace1}
  \min_{\mathcal{V}^T * \mathcal{V} = \mathcal{I}}
    \operatorname{Trace}_f\left( \mathcal{V}^T * \mathcal{A} * \mathcal{V} \right)
  = \frac{1}{p}\sum_{i=1}^p \sum_{j=1}^{d} \widetilde{\lambda}_j(\widehat{A}_i),
\end{equation}
where, for each $i=1,\ldots,p$, 
$\widetilde{\lambda}_1(\widehat{A}_i)\le \cdots\le \widetilde{\lambda}_m(\widehat{A}_i)$ are the eigenvalues of $\widehat{A}_i$, and 
$\widetilde{\lambda}_j(\widehat{A}_i)$, $j=1,\ldots,d$, denote the $d$ smallest eigenvalues of $\widehat{A}_i$.
\end{theorem}

\begin{proof}
The proof is analogous to that of Theorem~\ref{theotrace}, except that in the matrix problems 
\[
  \min_{\widehat{V}_i^T \widehat{V}_i = I_d}
    \operatorname{Trace}(\widehat{V}_i^T \widehat{A}_i \widehat{V}_i),
\]
the optimal value is given by the sum of the $d$ smallest eigenvalues of 
$\widehat{A}_i$. Summing over $i$ and using the definition of 
$\operatorname{Trace}_f$ gives \eqref{mintrace1}.
\end{proof}

We note in passing that the minimization problem that is obtained by replacing maximization in \eqref{maxtracegen1} by minimization can be solved by selecting the $d$ smallest generalized eigenvalues of each pair 
$(\widehat{A}_i,\widehat{B}_i)$, $i=1,\ldots,p$.

\section{Multidimensional Principal Component Analysis (MPCA) via the t-product}\label{sec:MPCA}
Consider a data set of lateral slices 
$\vec{\mathcal X}_1,\ldots, \vec{\mathcal X}_n \in {\mathbb R}^{m \times 1 \times p}$ and define
\[
  \mathcal X=\bigl[\vec{\mathcal X}_1, \ldots,\vec{\mathcal X}_n\bigr] 
  \in {\mathbb R}^{m \times n \times p}.
\]
Let
\[
  \mathcal V = \bigl[\vec{\mathcal V}_1, \ldots,\vec{\mathcal V}_d\bigr] 
  \in {\mathbb R}^{m \times d \times p},
  \qquad 1\le d \ll m,
\]
be a tensor that defines a mapping to a lower-dimensional subspace. We discuss the choice of this tensor below. The tensor $\mathcal Y$ obtained by mapping $\mathcal X$ to a lower-dimensional space is given by
\begin{equation}\label{proj1}
  \mathcal Y= \mathcal V^T * \mathcal X \in {\mathbb R}^{d \times n \times p}.
\end{equation}

Linear dimensionality reduction techniques apply a linear transformation to the high-dimensional data $\mathcal X$ to obtain a lower-dimensional representation that preserves specific desired properties of the original data. The tensor $\mathcal V$ in \eqref{proj1} is determined by solving an optimization problem designed to preserve, in the reduced space, either the variance of the data, the neighborhood structure, or some supervised discriminative criterion. Different objectives lead to different methods. This section focuses on an unsupervised variance-maximizing method that yields MPCA.

Once the mapping tensor $\mathcal V$ is determined, each lateral slice 
$\vec{\mathcal X}_i$ is mapped according to
\[
  \vec{\mathcal Y}_i = \mathcal V^T * \vec{\mathcal X}_i,\quad i=1,\ldots,n.
\]
If $\mathcal V$ is f-orthogonal, that is, if $\mathcal V^T * \mathcal V = \mathcal I$, then $\mathcal Y$ represents an f-orthogonal projection of 
$\mathcal X$ onto the subspace generated by the lateral slices
$\{\vec{\mathcal V}_1,\ldots,\vec{\mathcal V}_d\}$.

In the matrix case, Principal Component Analysis (PCA) is a classical unsupervised method that transforms a set of correlated variables into a new set of linearly uncorrelated variables referred to as principal components. They are ordered so that the first components capture most of the variance of the data, while components associated with smaller variance often are discarded to achieve dimensionality reduction. The principal components are uncorrelated but not necessarily statistically independent, unless the data are Gaussian; see, for example, \cite{Jolliffe2002} for a discussion.

The tensor analogue considered here is a linear mapping with the aim to compute an f-orthogonal tensor 
$\mathcal V\in{\mathbb R}^{m\times d\times p}$, with $1\le d\ll m$, that maximizes the variance of the mapped slices. Let
\[
  \vec{\mathcal M}=\frac{1}{n} \sum_{j=1}^n \vec{\mathcal X}_j
\]
denote the mean of the data and define the centered tensor
\[
  \overline{\mathcal X}
  =\bigl[\vec{\mathcal X}_1-\vec{\mathcal M},\ldots,\vec{\mathcal X}_n-\vec{\mathcal M}\bigr]
  \in {\mathbb R}^{m \times n \times p}.
\]
The variance of the mapped data can be expressed as
\[
  \sum_{i=1}^n \left\Vert 
    \mathcal V^T *\bigl(\vec{\mathcal X}_i-\vec{\mathcal M}\bigr)
  \right\Vert_F^2
  = \sum_{i=1}^n \left\Vert 
    \mathcal V^T *\vec{\overline{\mathcal X}}_i
  \right\Vert_F^2,
\]
where $\vec{\overline{\mathcal X}}_i$ denotes the $i$th lateral slice of $\overline{\mathcal X}$. MPCA seeks to solve the optimization problem
\begin{equation}\label{tpca11}
  \max_{\mathcal V^T * \mathcal V=\mathcal I}
    \sum_{i=1}^n 
      \left\Vert \mathcal V^T *\vec{\overline{\mathcal X}}_i \right\Vert_F^2.
\end{equation}
Using the Frobenius inner product and Proposition~\ref{pro1}, we can write the sum in \eqref{tpca11} as a tensor trace:
\[
  \sum_{i=1}^n 
    \left\Vert \mathcal V^T *\vec{\overline{\mathcal X}}_i \right\Vert_F^2
  = \sum_{i=1}^n 
      \left\langle 
        \mathcal V^T *\vec{\overline{\mathcal X}}_i, 
        \mathcal V^T *\vec{\overline{\mathcal X}}_i
      \right\rangle
  = \operatorname{Trace}_f\bigl(
      \mathcal V^T * \overline{\mathcal X} * \overline{\mathcal X}^T * \mathcal V
    \bigr).
\]
Therefore the MPCA optimization problem \eqref{tpca11} can be written as
\begin{equation}\label{tpca2}
  \max_{\mathcal V^T * \mathcal V=\mathcal I}
    \operatorname{Trace}_f\bigl(
      \mathcal V^T * \overline{\mathcal X} * \overline{\mathcal X}^T * \mathcal V
    \bigr).
\end{equation}

Introduce the tensor
\[
  \mathcal A
    = \overline{\mathcal X} * \overline{\mathcal X}^T
    \in {\mathbb R}^{m \times m \times p}.
\]	
Then $\mathcal A$ is f-symmetric and f-positive semidefinite, and the MPCA problem \eqref{tpca2} can be expressed as
\[
  \max_{\mathcal V^T * \mathcal V=\mathcal I}
    \operatorname{Trace}_f\bigl(
      \mathcal V^T * \mathcal A * \mathcal V
    \bigr).
\]

\nd By Theorem~\ref{theotrace}, its optimal value is
\[
  \max_{\mathcal V^T * \mathcal V=\mathcal I}
    \operatorname{Trace}_f\bigl(
      \mathcal V^T * \mathcal A * \mathcal V
    \bigr)
  = \frac{1}{p}\sum_{i=1}^p \sum_{j=1}^{d} \lambda_j(\widehat{A}_i),
\]
where $\widehat{\mathcal A} = {\tt fft}(\mathcal A,[],3)$ and 
$\widehat{A}_i$ denotes the $i$th frontal slice, with its eigenvalues ordered according to
$\lambda_1(\widehat{A}_i)\ge \ldots\ge\lambda_m(\widehat{A}_i)$.

\section{Multidimensional orthogonal neighborhood preserving projections}\label{sec:MONPP}
Multidimensional Orthogonal Neighborhood Preserving Projection (MONPP) extends the matrix-based ONPP method in  \cite{Kokiopoulou2007} to tensor-valued data. The aim of this projection (or mapping) is to reduce the dimensionality of high-order datasets through a linear tensor mapping while preserving local neighborhood relationships. A key feature inherited from ONPP is that, after mapping, each sample must be reconstructed from its neighbors using exactly the same reconstruction weights as in the original space.

\subsection{Description of the method}
Let 
\[
\mathcal X = [\vec{\mathcal X}_1,\ldots,\vec{\mathcal X}_n] \in \mathbb{R}^{m \times n \times p},
\]
where each lateral slice $\vec{\mathcal X}_\ell \in \mathbb{R}^{m \times 1 \times p}$ represents one sample. Let $\mathcal V \in \mathbb{R}^{m \times d \times p}$, with $1\le d \ll m$, denote the mapping tensor, and let $\mathcal W \in \mathbb{R}^{n \times n \times p}$ be a third-order tensor of reconstruction weights. The mapped data are given by
\[
\mathcal Y = \mathcal V^T * \mathcal X \in \mathbb{R}^{d \times n \times p}.
\]
In analogy with matrix ONPP, MONPP seeks to determine the tensor $\mathcal V$ by solving the constrained minimization problem
\begin{equation}\label{monpp1}
\min_{\mathcal V^T * \mathcal V = \mathcal I} \left\| \mathcal V^T * \mathcal X * \bigl(\mathcal I - \mathcal W^T\bigr) \right\|_F^2,
\end{equation}
where $\mathcal I \in \mathbb{R}^{n \times n \times p}$ is the identity tensor and the t-transpose is taken in the sense of Definition~2.1.
Using the identity $\|\mathcal Z\|_F^2 = \operatorname{Trace_f}(\mathcal Z^T * \mathcal Z)$ and properties of the t-product, the objective function \eqref{monpp1} can be written as
\[
\begin{aligned}
\left\|\mathcal V^T * \mathcal X * (\mathcal I - \mathcal W^T)\right\|_F^2
&= \operatorname{Trace_f}\left[ \bigl(\mathcal V^T * \mathcal X * (\mathcal I - \mathcal W^T)\bigr)^T * \mathcal V^T * \mathcal X * (\mathcal I - \mathcal W^T)\right] \\
&= \operatorname{Trace_f}\left[ (\mathcal I - \mathcal W) * \mathcal X^T * \mathcal V * \mathcal V^T * \mathcal X * (\mathcal I - \mathcal W^T)\right].
\end{aligned}
\]
Using the constraint $\mathcal V^T * \mathcal V = \mathcal I$, this expression simplifies to
\[
\left\|\mathcal V^T * \mathcal X * (\mathcal I - \mathcal W^T)\right\|_F^2
= \operatorname{Trace_f}\left( \mathcal V^T * \mathcal M * \mathcal V\right),
\]
where
\begin{equation}\label{monpp4}
\mathcal M = \mathcal X * (\mathcal I - \mathcal W^T) * (\mathcal I - \mathcal W) * \mathcal X^T \in \mathbb{R}^{m \times m \times p}.
\end{equation}
Hence, the MONPP problem can be expressed as
\begin{equation}\label{monpp3}
\min_{\mathcal V^T * \mathcal V = \mathcal I} \operatorname{Trace_f}\left( \mathcal V^T * \mathcal M * \mathcal V \right).
\end{equation}
The tensor $\mathcal M$ is f-symmetric and f-positive-semidefinite. Transformation to the Fourier domain gives the tensors $\widehat{\mathcal X} = {\tt fft}(\mathcal X,[\,],3)$ and $\widehat{\mathcal W} = {\tt fft}(\mathcal W,[\,],3)$. Let $\widehat X_i$ and $\widehat W_i$ denote the frontal slices of $\widehat{\mathcal X}$ and $\widehat{\mathcal W}$, respectively. Then the frontal slices of $\widehat{\mathcal M} = {\tt fft}(\mathcal M,[\,],3)$ are given by
\[
\widehat M_i = \widehat X_i (I - \widehat W_i^T)(I - \widehat W_i)\widehat X_i^T,\qquad i=1,\ldots,p.
\]
For any vector $z$, we have
\[
z^T \widehat M_i z = \left\|(I - \widehat W_i)\widehat X_i^T z\right\|_2^2 \ge 0.
\]
Since $\widehat M_i^T = \widehat M_i$, it follows that $\mathcal M$ is f-symmetric and f-positive-semidefinite in the sense of Definition~\ref{def2.3}. It follows that the minimization problem \eqref{monpp3} is a special case of Theorem~\ref{theotracemin} with $\mathcal A = \mathcal M$. We obtain
\begin{equation}\label{mintrace}
\min_{\mathcal V^T * \mathcal V = \mathcal I} \operatorname{Trace_f} \left( \mathcal V^T * \mathcal M * \mathcal V \right) 
= \frac{1}{p} \sum_{i=1}^p \sum_{j=1}^d \widetilde \lambda_j(\widehat M_i)
= \frac{1}{p} \sum_{j=1}^d \left\|\widetilde{\bm \lambda}_j\right\|_1,
\end{equation}
where the $\widetilde \lambda_j(\widehat M_i)$, $j=1,\ldots,d$, denote the $d$ smallest eigenvalues of the matrix $\widehat M_i$ and 
\[
\left\|\widetilde{\bm \lambda}_j\right\|_1 = \sum_{i=1}^p \widetilde \lambda_j(\widehat M_i).
\]
It follows that a tensor $\mathcal V$ that solves \eqref{monpp3} can be determined by letting, for each $i=1,\ldots,p$, the frontal slice $\widehat V_i$ of $\widehat{\mathcal V} = {\tt fft}(\mathcal V,[\,],3)$ be formed by the eigenvectors associated with the $d$ smallest eigenvalues of $\widehat M_i$.

It remains to construct the weight tensor $\mathcal W$. This can be done similarly as for the matrix ONPP by using a locally linear embedding. Given the data tensor $\mathcal X$, one first defines, for each sample index $\ell$, a set of $k$ nearest neighbors within an appropriate distance, for example by using the Frobenius norm of the corresponding lateral slices. Reconstruction weights are then determined so that each sample is approximated by a linear combination of its neighbors with coefficients that sum to one, and are set to zero for non-neighbors. The matrices of weights $\widehat W_i$ that appear in 
\[
\widehat M_i = \widehat X_i (I - \widehat W_i^T)(I - \widehat W_i)\widehat X_i^T,\qquad i=1,\ldots,p,
\]
are computed according to the procedure described in \cite{Kokiopoulou2007} and applied independently to each frontal slice 
$\widehat X_i$ of the tensor $\widehat{\mathcal X} = {\tt fft}(\mathcal X,[\,],3)$. Once the slices $\widehat W_i$ have been determined, the inverse Fourier transform yields the tensor $\mathcal W$, and the generalized eigenproblems associated with the matrices $\widehat M_i$ completely characterize the MONPP mapping tensor $\mathcal V$.

\section{Multidimensional kernel-based methods}\label{sec:MKBM}
Kernel methods achieve nonlinear dimensionality reduction by implicitly mapping the data into a feature space of high (possibly infinite) dimension. Linear techniques such as PCA can then be applied effectively to elements in the feature space; see, e.g., \cite{Muller2001,Shawe2004,Scholkopf1998}. 

In our tensor setting, we consider an implicitly defined nonlinear feature map
\[
\Phi: \mathbb{R}^{m \times 1 \times p} \longrightarrow \mathcal H,
\]
where $\mathcal H$ is a real Hilbert space (the feature space). Given data slices
\[
\vec{\mathcal X}_1,\ldots, \vec{\mathcal X}_n \in \mathbb{R}^{m \times 1 \times p},
\]
we collect them into the tensor
\[
\mathcal X=[\vec{\mathcal X}_1, \ldots,\vec{\mathcal X}_n] \in \mathbb{R}^{m \times n \times p}.
\]
The feature-space representation of the data then is
\[
\Psi = [\Phi(\vec{\mathcal X}_1),\ldots,\Phi(\vec{\mathcal X}_n)],
\]
which we may view as a third-order tensor with $n$ samples living in $\mathcal H$ along the first mode. The mapping $\Phi$ is not explicitly computed. Instead, all computations are expressed in terms of the kernel function
\[
K(\vec{\mathcal X}_i,\vec{\mathcal X}_j)
= \left\langle \Phi(\vec{\mathcal X}_i), \Phi(\vec{\mathcal X}_j) \right\rangle_{\mathcal H},
\qquad i,j=1,\ldots,n,
\]
which defines the Gram matrix 
\[
K = [K_{ij}]_{i,j=1}^n \in \mathbb{R}^{n\times n}, 
\quad K_{ij}=K(\vec{\mathcal X}_i,\vec{\mathcal X}_j),
\]
which is symmetric and positive semidefinite. This matrix completely characterizes the geometry of the data in the feature space $\mathcal H$.
A common kernel choice is the Gaussian (RBF) kernel:
\begin{equation}\label{ker:gaussian}
K(\vec{\mathcal X}_i,\vec{\mathcal X}_j)
= \exp\!\left( - \frac{\|\vec{\mathcal X}_i-\vec{\mathcal X}_j\|_F^2}{2\sigma^2} \right),
\end{equation}
or, more generally, a radial basis function (RBF) kernel of the form
\begin{equation}\label{ker:rbf}
K(\vec{\mathcal X}_i,\vec{\mathcal X}_j)
= \exp\!\left( - \frac{\|\vec{\mathcal X}_i-\vec{\mathcal X}_j\|_F^2}{c} \right),
\end{equation}
for some parameter $c>0$. Other kernels (polynomial, Laplacian, sigmoid, etc.) also can be used; the only requirement is that $K$ be positive semidefinite so that $K_{ij}=K(\vec{\mathcal X}_i,\vec{\mathcal X}_j)$
corresponds to a valid inner product between ${\mathcal X}_i$ and ${\mathcal X}_j$ in a Hilbert space.
\medskip

We now interpret our tensor-based linear methods (MPCA and MONPP) in a kernel framework. Conceptually, one replaces the original data tensor 
$\mathcal X$ by its lifted representation $\Psi$ and then applies a t-product–based mapping in the feature space. The mapped data take the form
\[
\mathcal Y = \mathcal V^{T} * \Psi,
\]
where $\mathcal V$ is a mapping tensor that solves an appropriate optimization problem. However, since the feature-space dimension is typically very large (and often not explicitly known), we express the optimization problem in terms of the Gram matrices built from kernel evaluations. This avoids the explicit computation of $\Psi$.

\subsection{Multidimensional Kernel PCA (MKPCA)}
Multidimensional Kernel PCA (MKPCA) is a nonlinear kernel extension of MPCA. It carries out MPCA in the feature space $\mathcal H$ and is useful for visualization, classification, denoising, and feature extraction when the data possess nonlinear structure.
We first center the data in feature space. Let
\[
\vec{\mathcal M} = \frac{1}{n}\sum_{i=1}^n \vec{\mathcal X}_i
\]
denote the mean slice, and define the centered tensor
\[
\overline{\mathcal X}
= [\overline{\vec{\mathcal X}}_1,\ldots,\overline{\vec{\mathcal X}}_n], 
\qquad \overline{\vec{\mathcal X}}_i = \vec{\mathcal X}_i - \vec{\mathcal M}.
\]
In feature space, we consider
\[
\overline{\Psi} = [\Phi(\overline{\vec{\mathcal X}}_1),\ldots,\Phi(\overline{\vec{\mathcal X}}_n)],
\]
which is not explicitly formed; instead, we use the centered Gram matrices obtained from \eqref{ker:gaussian} or \eqref{ker:rbf}.
The objective of MKPCA is to solve
\begin{equation}\label{eq:mkpca_obj}
\max_{\mathcal V^T * \mathcal V = \mathcal I}
\operatorname{Trace_f}\!\bigl(
\mathcal V^T * \overline{\Psi} * \overline{\Psi}^T * \mathcal V
\bigr).
\end{equation}

Introduce the feature-space covariance tensor
\[
\mathcal A = \overline{\Psi} * \overline{\Psi}^T
\]
Similarly as in the linear MPCA case, $\mathcal A$ is f-symmetric and f-positive-semidefinite. Therefore Theorem~\ref{theotrace} applies. Define
\[
\widehat{\mathcal A} = {\tt fft}(\mathcal A,[\,],3)
\]
and let $\widehat A_i$ be the $i$th frontal slice. The optimization problem \eqref{eq:mkpca_obj} splits into $p$ independent matrix optimization problems
\begin{equation}\label{eq:mkpca_slice_trace}
\max_{\widehat V_i^T \widehat V_i = I}
\mathrm{Trace}\bigl(\widehat V_i^T \widehat A_i \widehat V_i\bigr),
\qquad i=1,\ldots,p,
\end{equation}
whose solutions are obtained by letting $\widehat V_i$ be the matrix of eigenvectors associated with the $d$ largest eigenvalues of $\widehat A_i$. Equivalently, if we write
\[
\widehat A_i = \widehat{\overline{\Psi}}_i\,\widehat{\overline{\Psi}}_i^T,
\]
where $\widehat{\overline{\Psi}}_i$ is the $i$th frontal slice of $\widehat{\overline{\Psi}}={\tt fft}(\overline{\Psi},[\,],3)$. Then the eigenvalues of $\widehat A_i$ are the squares of the singular values of $\widehat{\overline{\Psi}}_i$:
\begin{equation}\label{eq:mkpca_eigs}
\max_{\widehat V_i^T \widehat V_i = I}
\mathrm{Trace}\bigl(\widehat V_i^T \widehat A_i \widehat V_i\bigr)
= \sum_{j=1}^d \lambda_j(\widehat A_i)
= \sum_{j=1}^d \bigl(\sigma_j^{(i)}(\widehat{\overline{\Psi}}_i)\bigr)^2,
\end{equation}
where $\lambda_1(\widehat A_i)\ge\cdots\ge\lambda_d(\widehat A_i)$ denote the $d$ largest eigenvalues, and $\sigma_j^{(i)}(\widehat{\overline{\Psi}}_i)$ are the corresponding singular values.

In actual computations, $\widehat{\overline{\Psi}}_i$ is not available, however the Gram matrix
\[
\overline G_i = \widehat{\overline{\Psi}}_i^T \widehat{\overline{\Psi}}_i \in \mathbb{R}^{n\times n}
\]
can be computed by kernel evaluations:
\[
(\overline G_i)_{lq}
= \left\langle \Phi(\overline{\vec{\mathcal X}}_l), \Phi(\overline{\vec{\mathcal X}}_q)\right\rangle_{\mathcal H}
= K(\overline{\vec{\mathcal X}}_l,\overline{\vec{\mathcal X}}_q),
\qquad l,q = 1,\ldots,n.
\]
As in standard kernel PCA, we exploit the relation
\[
\widehat{\overline{\Psi}}_i \widehat{\overline{\Psi}}_i^T u = \lambda u
\quad\Longrightarrow\quad
\overline G_i v = \lambda v, \quad v = \widehat{\overline{\Psi}}_i^T u,
\]
so that the spectrum of $\widehat A_i$ can be computed from the $n\times n$ matrix $\overline G_i$. The centered Gram matrix $\overline G_i$ is determined from an uncentered Gram matrix $G_i$ by
\[
\overline G_i = H G_i H, \qquad
H = I_n - \frac{1}{n}\mathbf 1 \mathbf 1^T,
\]
with $G_i$ defined by \eqref{ker:gaussian} (or any other choice of kernel).

We summarize the computations required for MKPCA: For each $i=1,\ldots,p$: 
\begin{enumerate}
  \item Build the centered Gram matrix $\overline G_i$ using the kernel $K$.

  \item Compute the $d$ largest eigenvalues and associated eigenvectors 
  $v_1^{(i)},\ldots,v_d^{(i)}$ of $\overline G_i$.

  \item Form the slice of the mapped data as
  \[
  \widehat{\mathcal Y}(:,:,i) = 
  \begin{bmatrix}
  (v_1^{(i)})^T\\
  \vdots\\
  (v_d^{(i)})^T
  \end{bmatrix},
  \]
  which corresponds to the coordinates of the $n$ samples along the first $d$ kernel principal components of the $i$th frontal slice.

  \item Apply the inverse FFT along the third mode to recover the mapped tensor
  \[
  \mathcal Y = {\tt ifft}(\widehat{\mathcal Y},[\,],3).
  \]
\end{enumerate}

\subsection{Multidimensional Kernel Orthogonal Beighborhood Preserving Projections (MKONPP)}
We extend MONPP to a kernel setting, in the same spirit as MKPCA. Instead of applying ONPP directly to $\mathcal X$, we apply it in feature space to $\overline{\Psi}$. Recall that in the linear case, MONPP minimizes
\[
\left\| \mathcal V^T * \mathcal X * (\mathcal I - \mathcal W^T) \right\|_F^2
\]
under the constraint $\mathcal V^T * \mathcal V = \mathcal I$, where $\mathcal W \in \mathbb{R}^{n\times n\times p}$ is the tensor of reconstruction weights. In the kernel setting, we consider the centered feature-space data $\overline{\Psi}$ and define the MKONPP objective function
\begin{equation}\label{eq:mkonpp_obj}
\min_{\mathcal V^T * \mathcal V = \mathcal I}
\left\| \mathcal V^T * \overline{\Psi} * (\mathcal I - \mathcal W^T) \right\|_F^2.
\end{equation}
Using that $\|\mathcal Z\|_F^2 = \operatorname{Trace_f}(\mathcal Z^T * \mathcal Z)$ and the properties of the t-product, the objective function \eqref{eq:mkonpp_obj} can be written as
\[
\operatorname{Trace_f}\!\left(
\mathcal V^T * \mathcal M^{(\mathrm{ker})} * \mathcal V
\right),
\]
where
\begin{equation}\label{eq:mkonpp_Mker}
\mathcal M^{(\mathrm{ker})}
= \overline{\Psi} * \mathcal Q * \overline{\Psi}^T,
\qquad
\mathcal Q = (\mathcal I - \mathcal W^T)*(\mathcal I - \mathcal W),
\end{equation}
and $\mathcal I$ is the identity tensor of size $(n,n,p)$. As in the linear case, $\mathcal Q$ is f-symmetric and f-positive-semidefinite, and so is $\mathcal M^{(\mathrm{ker})}$. The minimization problem \eqref{eq:mkonpp_obj} therefore is equivalent to
\begin{equation}\label{eq:mkonpp_trace}
\min_{\mathcal V^T * \mathcal V = \mathcal I}
\operatorname{Trace_f}\!\left(
\mathcal V^T * \mathcal M^{(\mathrm{ker})} * \mathcal V
\right),
\end{equation}
which is a direct instance of Theorem~\ref{theotracemin}. Let
\[
\widehat{\mathcal M}^{(\mathrm{ker})}
= {\tt fft}(\mathcal M^{(\mathrm{ker})},[\,],3),
\]
and denote its frontal slices by $\widehat M^{(\mathrm{ker})}_i$, $i=1,\ldots,p$. Then
\[
\widehat M^{(\mathrm{ker})}_i
= \widehat{\overline{\Psi}}_i \, \widehat Q_i \, \widehat{\overline{\Psi}}_i^T,
\qquad i=1,\ldots,p,
\]
where $\widehat{\overline{\Psi}}_i$ and $\widehat Q_i$ are the $i$th frontal slices of $\widehat{\overline{\Psi}}$ and $\widehat{\mathcal Q}$, respectively. The minimization problem \eqref{eq:mkonpp_trace} splits into $p$ matrix minimization problems
\[
\min_{\widehat V_i^T \widehat V_i = I}
\mathrm{Trace}\bigl(\widehat V_i^T \widehat M^{(\mathrm{ker})}_i \widehat V_i\bigr),
\quad i=1,\ldots,p,
\]
whose solutions are obtained by taking $\widehat V_i$ to be the matrix of eigenvectors associated with the $d$ smallest eigenvalues of $\widehat M^{(\mathrm{ker})}_i$. The corresponding minimal value is
\[
\min_{\mathcal V^T * \mathcal V = \mathcal I}
\operatorname{Trace_f}\!\left(
\mathcal V^T * \mathcal M^{(\mathrm{ker})} * \mathcal V
\right)
= \frac{1}{p} \sum_{i=1}^p \sum_{j=1}^d \widetilde\lambda_j(\widehat M^{(\mathrm{ker})}_i),
\]
where $\widetilde\lambda_1(\widehat M^{(\mathrm{ker})}_i)\le\cdots\le\widetilde\lambda_d(\widehat M^{(\mathrm{ker})}_i)$.

Similarly as in MKPCA, the $\widehat{\overline{\Psi}}_i$ are not explicitly available, but the Gram matrices
\[
\overline G_i = \widehat{\overline{\Psi}}_i^T \widehat{\overline{\Psi}}_i
\in \mathbb{R}^{n\times n}, \qquad i=1,\ldots,p,
\]
can be constructed from the kernel function using centered data:
\[
(\overline G_i)_{lq} 
= K(\overline{\vec{\mathcal X}}_l, \overline{\vec{\mathcal X}}_q),
\qquad l,q = 1,\ldots,n.
\]
Writing the eigenproblems in feature-space as
\[
\widehat M^{(\mathrm{ker})}_i u = \lambda u
\quad\Longleftrightarrow\quad
\widehat{\overline{\Psi}}_i \widehat Q_i \widehat{\overline{\Psi}}_i^T u = \lambda u,\qquad i=1,\ldots,p,
\]
and multiplying from the left by $\widehat{\overline{\Psi}}_i^T$, we obtain equivalent eigenproblems in the $n$-dimensional sample-space,
\[
\overline G_i \widehat Q_i v = \lambda v,
\qquad v = \widehat{\overline{\Psi}}_i^T u,\qquad i=1,\ldots,p.
\]
It follows that MKONPP can be implemented by solving, for each $i=1,\ldots,p$, an $n\times n$ eigenproblem involving the centered Gram matrix $\overline G_i$ and the weight matrix $\widehat Q_i$ constructed from $\mathcal W$.
\medskip

The weight tensor $\mathcal W$ is chosen analogously as for MONPP (and for the matrix ONPP of \cite{Kokiopoulou2007}). For each frontal slice $\widehat X_i$ of $\widehat{\mathcal X}={\tt fft}(\mathcal X,[\,],3)$ and each sample index $l \in \{1,\ldots,n\}$, we define a neighborhood $\mathcal N_i(l)$ of $k$ nearest neighbors using, for example, the Frobenius norm to measure the distance between columns,
\[
d_i(l,q) = \bigl\| \widehat X_i^{(l)} - \widehat X_i^{(q)} \bigr\|_F,
\]
where $\widehat X_i^{(l)}$ denotes the $l$th column of $\widehat X_i$. A common choice of the (unnormalized) affinity weights is
\begin{equation}\label{eq:mkonpp_weights}
\widehat w_{lq}^{(i)}
= \exp\!\left(
- \frac{\bigl\| \widehat X_i^{(l)} - \widehat X_i^{(q)} \bigr\|_F^2}{2\sigma^2}
\right),
\qquad l,q = 1,\ldots,n,
\end{equation}
with $\widehat w_{lq}^{(i)}=0$ if $q \notin \mathcal N_i(l)$, and $\sigma>0$ a scaling parameter. Each row is then normalized so that
\[
\sum_{q=1}^n \widehat w_{lq}^{(i)} = 1,\qquad l=1,\ldots,n.
\]
This ensures that each sample is reconstructed as a convex combination of its neighbours. The matrices $\widehat W_i = [\widehat w_{lq}^{(i)}]$ form the frontal slices of the tensor $\widehat{\mathcal W}$, from which $\mathcal W$ is obtained by the inverse FFT.

In summary, MKPCA and MKONPP extend MPCA and MONPP to nonlinear 
feature-spaces by combining the t-product framework with kernel methods. All computations are carried out via Gram matrices built from tensor kernels, thereby preserving the multilinear structure of the data while enabling nonlinear dimensionality reduction.

\section{Multidimensional nonlinear reduction methods}\label{sec:MNRM}
This section discusses two nonlinear reduction methods.

\subsection{Multidimensional Locally Linear Embedding (MLLE)}
Classical matrix Locally Linear Embedding (LLE) is a nonlinear dimensionality reduction technique designed to uncover meaningful low-dimensional structure hidden in high-dimensional datasets. The key assumption is that, although the data live in a high-dimensional space, they lie (approximately) on a smooth low-dimensional manifold. Rather than preserving global distances, LLE focuses on preserving local geometry. For each data point, the method first identifies its nearest neighbors in the original space and expresses it as a linear combination of those neighbors using
reconstruction weights. These weights describe how each point relates to its local manifold patch. In the second step, LLE finds a low-dimensional embedding, in which these same reconstruction relationships are preserved as well as possible: each embedded point should still be approximated by the same weighted combination of its neighbors.

We proceed as follows in the multidimensional (tensor) case: 
Given samples
\[
\vec{\mathcal X}_1,\ldots, \vec{\mathcal X}_n \in \mathbb{R}^{m \times 1 \times p}
\]
and the tensor
\[
\mathcal X=[\vec{\mathcal X}_1, \ldots,\vec{\mathcal X}_n] \in \mathbb{R}^{m \times n \times p},
\]
we seek a reduced tensor
\[
\mathcal Y \in  \mathbb{R}^{d \times n \times p}, \qquad 1\le  d \ll m,
\]
whose lateral slices contain the $d$-dimensional embedding of the $n$ samples for each mode.
MLLE seeks to solve the following minimization problem in tensor form:
\begin{equation}\label{lle1}
\min_{\mathcal{Y} * \mathcal{Y}^T = \mathcal{I}_d}
\; \operatorname{Trace}_f \!\left[
\mathcal{Y} * \left( \mathcal{I}-\mathcal{W}^T \right)
          * \left( \mathcal{I}-\mathcal{W} \right)
          * \mathcal{Y}^T
\right],
\end{equation}
where $\mathcal{W} \in \mathbb{R}^{n \times n \times p}$ is the tensor of
reconstruction weights associated with the data $\mathcal X$, and
$\mathcal{I}$ is the f-identity tensor of size $(n,n,p)$. The constraint
$\mathcal{Y} * \mathcal{Y}^T = \mathcal{I}_d$ (with
$\mathcal{I}_d \in \mathbb{R}^{d \times d \times p}$) avoids the trivial solution
$\mathcal{Y}=0$ and determines the overall scale of the embedding.
Define the f-symmetric and f-positive-semidefinite ``MLLE tensor''
\[
\mathcal{N} = \left( \mathcal{I}-\mathcal{W}^T \right)
              * \left( \mathcal{I}-\mathcal{W} \right)
\in \mathbb{R}^{n \times n \times p}.
\]
Then \eqref{lle1} can be written as
\[
\min_{\mathcal{Y} * \mathcal{Y}^T = \mathcal{I}_d}
\; \operatorname{Trace}_f \left( \mathcal{Y} * \mathcal{N} * \mathcal{Y}^T \right).
\]
Using the definition of $\operatorname{Trace}_f$ and the properties of the t-product, this tensor minimization problem splits in the Fourier domain into $p$ independent matrix minimization problems. Let
\[
\widehat{\mathcal{Y}} = {\tt fft}(\mathcal{Y},[\,],3),\quad
\widehat{\mathcal{N}} = {\tt fft}(\mathcal{N},[\,],3),
\]
and denote by $\widehat{Y}_i$ and $\widehat{N}_i$ the $i$th frontal slices of
$\widehat{\mathcal{Y}}$ and $\widehat{\mathcal{N}}$, respectively. Then
\[
\operatorname{Trace}_f \left( \mathcal{Y} * \mathcal{N} * \mathcal{Y}^T \right)
= \frac{1}{p} \sum_{i=1}^p
\operatorname{Trace} \left( \widehat{Y}_i \widehat{N}_i \widehat{Y}_i^T \right),
\]
and the minimization problem \eqref{lle1} can be expressed as
\begin{equation}\label{lle2}
\min_{\widehat{Y}_i \widehat{Y}_i^T = I_d}
\; \operatorname{Trace} \left(
\widehat{Y}_i \widehat{N}_i \widehat{Y}_i^T
\right),
\qquad i=1,\ldots,p.
\end{equation}
Since each $\widehat{N}_i$ is symmetric positive semidefinite, the solution of \eqref{lle2} can be obtained by solving the eigenproblems
\begin{equation}\label{lle3}
\widehat{N}_i v = \lambda v,\qquad i=1,\ldots,p,
\end{equation}
and selecting, for each $i$, the eigenvectors associated with the $d$ smallest
nonzero eigenvalues. More precisely, if we order the eigenvalues of
$\widehat{N}_i$ as
\[
0 = \lambda_1^{(i)} \le \lambda_2^{(i)} \le \cdots \le \lambda_n^{(i)},
\]
and denote the corresponding eigenvectors by $v_i^{(j)}$, $j=1,\ldots,n$, then the
$i$-th frontal slice of the embedding tensor is
\[
\widehat{Y}_i =
\begin{bmatrix}
\left(v_i^{(2)}\right)^T \\
\vdots \\
\left(v_i^{(d+1)}\right)^T
\end{bmatrix}
\in \mathbb{R}^{d \times n},\qquad i=1,\ldots,p.
\]
Finally, the multidimensional embedding is recovered in the original domain via
\[
\mathcal{Y} = {\tt ifft}(\widehat{\mathcal{Y}},[\,],3).
\]
The weights forming $\mathcal{W}$ are constructed as in the neighborhood-based schemes described earlier (e.g., by using Gaussian weights on $k$ nearest neighbors).

\subsection{Multidimensional Laplacian Eigenmaps (MLE)}
Laplacian Eigenmaps refers to a nonlinear dimensionality reduction method grounded in manifold learning and spectral graph theory. The method constructs a weighted graph,
whose nodes are the data samples and the edges have weights that encode local similarity. This determines a symmetric adjacency matrix associated with the graph. A graph Laplacian is built from this adjacency matrix, and a low-dimensional embedding is obtained from the eigenvectors associated with its smallest nontrivial eigenvalues. This embedding preserves local neighborhood structure and captures the intrinsic geometry of the data manifold. Dimensionality reduction of matrix-valued data with the aid of Laplacian Eigenmaps is described in \cite{belkin2001,belkin2003}.

In the multidimensional setting, let
$\mathcal X \in \mathbb{R}^{m \times n \times p}$
be the data tensor (with $n$ samples), let
$\mathcal W \in \mathbb{R}^{n \times n \times p}$ be an associated tensor of graph weights, and let $\mathcal{D} \in \mathbb{R}^{n \times n \times p}$ be a degree tensor. The tensors $\mathcal{W}$ and $\mathcal{D}$ will be defined below as inverse Fourier transforms of the computed tensors $\widehat{\mathcal{W}}$ and $\widehat{\mathcal{D}}$, respectively. Introduce the Laplacian tensor for the multilayer graph determined by ${\mathcal W}$ and ${\mathcal D}$ as
\begin{equation}\label{lap1}
\mathcal L = \mathcal D - \mathcal W
\in \mathbb{R}^{n \times n \times p}.
\end{equation}
The Multidimensional Laplacian Eigenmaps (MLE) method seeks an embedding tensor
\[
\mathcal Z \in \mathbb{R}^{d \times n \times p}, \qquad 1\le d \ll m,
\]
by solving
\begin{equation}\label{lap3}
\min_{\mathcal{Z} * \mathcal{D} * \mathcal{Z}^T = \mathcal{I}_d}
\operatorname{Trace}_f \left( \mathcal{Z} * \mathcal{L} * \mathcal{Z}^T \right),
\end{equation}
where the constraint $\mathcal{Z} * \mathcal{D} * \mathcal{Z}^T = \mathcal{I}_d$ provides a normalization. Using properties of the t-product together with the definition of $\operatorname{Trace}_f$, the minimization problem \eqref{lap3} decouples in the Fourier domain into
$p$ generalized eigenproblems: Let
\[
\widehat{\mathcal{L}} = {\tt fft}(\mathcal L,[\,],3),\quad
\widehat{\mathcal{D}} = {\tt fft}(\mathcal D,[\,],3),\quad
\widehat{\mathcal{Z}} = {\tt fft}(\mathcal Z,[\,],3),
\]
and denote the $i$th frontal slices of the transformed tensors by 
$\widehat{L}_i$, $\widehat{D}_i$, and $\widehat{Z}_i$, respectively. Then the minimization problem \eqref{lap3} is equivalent to the $p$ independent minimization problems
\[
\min_{\widehat{Z}_i \widehat{D}_i \widehat{Z}_i^T = I_d}
\operatorname{Trace}\big( \widehat{Z}_i \widehat{L}_i \widehat{Z}_i^T \big),
\qquad i=1,\ldots,p,
\]
whose solutions are given by solving the generalized eigenvalue problems
\begin{equation}\label{lap5}
\widehat{L}_i \, v_i^{(j)} = \lambda_j^{(i)} \, \widehat{D}_i \, v_i^{(j)},
\qquad i=1,\ldots,p.
\end{equation}

For each $i$, we order the generalized eigenvalues as
\[
0 = \lambda_1^{(i)} \le \lambda_2^{(i)} \le \cdots \le \lambda_n^{(i)},
\]
and discard the eigenvector associated with 
$\lambda_1^{(i)}=0$. A
solution for the $i$th frontal slice of the embedding tensor is obtained as
\begin{equation}\label{lap6}
\widehat{Z}_i =
\begin{bmatrix}
\big(v_i^{(2)}\big)^T \\
\vdots \\
\big(v_i^{(d+1)}\big)^T
\end{bmatrix},
\qquad i=1,\ldots,p.
\end{equation}
Finally, the embedding in the original domain is given by
\[
\mathcal{Z} = {\tt ifft}(\widehat{\mathcal{Z}},[\,],3).
\]

For each frontal slice $i=1,\ldots,p$, the matrices $\widehat{L}_i$ and
$\widehat{D}_i$ satisfy
\begin{equation}\label{hatLi}
\widehat{L}_i = \widehat{D}_i - \widehat{W}_i,
\end{equation}
where $\widehat{W}_i$ is the $i$th frontal slice of
$\widehat{\mathcal{W}} = {\tt fft}(\mathcal W,[\,],3)$. A standard and practical
choice for the weights are the Gaussian kernel on pairwise distances:
\begin{equation}\label{lap9}
\widehat{W}_{lq i}
= \exp\!\left(
   - \frac{\big\|\widehat{\vec X}_i^{(l)} - \widehat{\vec X}_i^{(q)}\big\|_F^2}
          {2 \sigma^2}
 \right),
\quad l,q = 1,\ldots,n,\;\; i=1,\ldots,p,
\end{equation}
where $\widehat{\vec X}_i^{(l)}$ denotes the $l$th column of the $i$th frontal slice
$\widehat{X}_i$ of the tensor $\widehat{\mathcal X} = {\tt fft}(\mathcal X,[\,],3)$,
and $\sigma > 0$ is a scaling parameter that controls the locality of the graph connections. Thus, given the data tensor 
$\mathcal X \in \mathbb{R}^{m \times n \times p}$, we compute its transform $\widehat Z$ and the tensor \eqref{lap9}. Let the diagonal matrices $\widehat D_i$, i=1,\ldots,p, be made up of the diagonal entries of the symmetric matrix $\widehat W_i$ and define the matrices 
$\widehat L_i$ by \eqref{hatLi}. Then the matrices $\widehat Z_i$ are
determined by \eqref{lap6} after solving the generalized eigenproblems
\eqref{lap5}.

\section{Numerical experiments}\label{sec:NE}
This section illustrates the performance of the proposed multilinear dimensionality reduction methods when applied to  several benchmark multidimensional datasets. The general procedure is to first project the original high-dimensional tensor data onto a low-dimensional subspace, and then classify the reduced representations using standard supervised learning algorithms. We report both classification accuracy and computing time and
compare the performance of Multidimensional Principal Component Analysis (MPCA) and its kernel extension (MKPCA),
Multidimensional Orthogonal Neighborhood Preserving Projection (MONPP) and its kernel variant (MKONPP), Multidimensional Locally Linear Embedding (MLLE), and the tensor version of Laplacian Eigenmaps (MLE).

Given a reduced tensor
\[
\mathcal{Y} \in \mathbb{R}^{d \times n \times p},
\]
we construct a matrix suitable for conventional classifiers by vectorizing each
sample across the first and third modes. More precisely, the $j$th sample is
mapped to a vector in $\mathbb{R}^{dp}$, and all samples are stacked to form a
matrix in $\mathbb{R}^{n \times dp}$. The classifier is then trained and evaluated
on this matrix representation.
We choose the bandwidth parameter for the kernel methods 
\[
\sigma = \sqrt{\frac{m_d}{2}},
\] 
where $m_d$ denotes the median of all pairwise squared distances between the $n$ samples.

Each dimensionality reduction method is evaluated for two target dimensions,
$d = 5$ and $d = 10$. For classification, we use four widely adopted classifiers:
Random Forest (RF) \cite{Breiman2001RandomForests}, K-Nearest Neighbors (KNN)
\cite{CoverHart67}, Support Vector Machine (SVM) \cite{svm95}, and Extra Trees
\cite{Etrees}. Performance is measured using stratified 5-fold cross-validation:
in each fold, 80\% of the data are used for training and 20\% for testing. The
classification accuracy is defined as
\begin{equation}
\text{Accuracy}(\%)= \frac{\text{Number of correct predictions}}{\text{Total number of samples}} \times 100.
\end{equation}
\subsection{Datasets}

\nd The following benchmark datasets are used in our experiments:
\medskip
\begin{itemize}
   \item \textbf{Salinas} \cite{salinas-dataset}: 
    A hyperspectral dataset with 224 spectral bands and 16 agricultural classes. 
    After removing water absorption bands, 204 bands are retained and labeled pixels
    are extracted. The tensor representation follows the same spatial–spectral
    unfolding strategy as in standard hyperspectral classification settings.

    \item \textbf{FSDD} \cite{fsdd}: 
    The Free Spoken Digit Dataset consists of audio recordings of digits (0--9) from 6 speakers. 
    Time–frequency representations (spectrograms) of size $64 \times 64 \times 4$ 
    are computed for 2,000 signals and reshaped into tensors of size $(4096, 2000, 4)$.

     \item \textbf{Gait Recognition} \cite{casia-gait}: 
    750 binary silhouette sequences from 10 individuals performing walking motions.
    Each $128 \times 128$ silhouette image is partitioned into horizontal bands to
    form third-order tensors of size $(8192, 750, 2)$.

    \item \textbf{Washington DC Mall (WDCM)}:
    Hyperspectral images of the Washington DC Mall collected by the HYDICE sensor \footnote{\url{http://lesun.weebly.com/hyperspectral-data-set.html}}. The full scene contains $1208 \times 307$ pixels and 191 spectral
    bands. Four classes (\emph{grass land}, \emph{tree}, \emph{roof}, and \emph{road})
    are extracted from manually selected $7 \times 7$ blocks, and reshaped into
    tensors of size $(4704, 2000, 2)$.

    \item \textbf{MNIST} \cite{mnist-dataset}: 
    A dataset of handwritten digits (0--9) with grayscale images of size $28 \times 28$. 
    For tensor-based processing, each image is partitioned into small patches to form
    tensors of size \hfill\break $(392, 3000, 2)$.
\end{itemize}

\nd Table~\ref{tabledatasets} summarizes the tensor structures and domain
characteristics of the datasets.

\begin{table}[h!]
\centering
\caption{Dataset specifications, tensor dimensions, and class distributions.}
\begin{tabular}{lcccl}
\hline
\textbf{Dataset} & \textbf{Tensor} & \textbf{Samples} & \textbf{Classes} & \textbf{Domain / Modality} \\
\hline
Salinas          & (102, 2000, 2)   & 2000 & 16  & Hyperspectral imagery \\
FSDD             & (4096, 2000, 4)  & 2000 & 10  & Audio / spoken digits \\
Gait             & (8192, 750, 2)   & 750  & 10  & Human motion silhouettes \\
Washington DC    & (4704, 2000, 2)  & 2000 & 4   & Hyperspectral imagery \\
MNIST            & (392, 2000, 2)   & 3000 & 10  & Handwritten digits  \\
\hline
\end{tabular}
\label{tabledatasets}
\end{table}

\subsection{Results and discussion}

In what follows, we present quantitative results for each dataset. For brevity,
we report mean classification accuracies over the 5 folds and the corresponding
CPU-times required by each dimensionality reduction method (including projection
and classifier training).

\subsubsection{Salinas}

\begin{table}[htp]
\centering
\caption{Results for the Salinas dataset.}
\begin{tabular}{lcccccc}
\hline
\textbf{d} 
& \textbf{MPCA} 
& \textbf{MKPCA} 
& \textbf{MONPP} 
& \textbf{MKONPP} 
& \textbf{MLLE} 
& \textbf{MLE} \\
\hline
\multicolumn{7}{c}{\textbf{RF - Accuracy (\%)}} \\
\hline
$d = 5$  & 91.85 & 85.80 & 82.55 & 61.50 & 85.00 & 85.95 \\
$d = 10$ & 92.65 & 88.95 & 84.35 & 64.70 & 87.60 & 86.15 \\
\hline
\multicolumn{7}{c}{\textbf{KNN - Accuracy (\%)}} \\
\hline
$d = 5$  & 90.65 & 83.65 & 80.00 & 59.65 & 83.80 & 83.90 \\
$d = 10$ & 90.30 & 87.10 & 81.65 & 64.60 & 85.80 & 84.60 \\
\hline
\multicolumn{7}{c}{\textbf{SVM - Accuracy (\%)}} \\
\hline
$d = 5$  & 92.60 & 86.60 & 83.35 & 64.75 & 78.20 & 83.10 \\
$d = 10$ & 92.90 & 90.40 & 85.90 & 69.35 & 86.25 & 85.60 \\
\hline
\multicolumn{7}{c}{\textbf{ExtraTrees - Accuracy (\%)}} \\
\hline
$d = 5$  & 92.85 & 86.35 & 82.30 & 62.55 & 84.70 & 85.70 \\
$d = 10$ & 92.55 & 88.65 & 84.35 & 66.10 & 87.45 & 86.30 \\
\hline
\multicolumn{7}{c}{\textbf{CPU-time (s)}} \\
\hline
$d = 5$  & 3.02 & 7.49 & 3.25 & 8.05 & 3.11 & 3.25 \\
$d = 10$ & 2.87 & 7.30 & 3.38 & 8.72 & 4.84 & 3.26 \\
\hline
\end{tabular}
\label{tab:salinas}
\end{table}

\nd Table~\ref{tab:salinas} shows that, for almost every classifier and both values
of $d$, MPCA attains the highest or near-highest accuracy. For instance, with SVM
and $d = 10$, MPCA reaches $92.90\%$, outperforming the alternatives for this
configuration. Increasing the dimensionality from $d = 5$ to $d = 10$ generally
improves accuracy for most methods (e.g., MKPCA + SVM: $86.60\% \rightarrow 90.40\%$),
indicating that a slightly richer subspace helps preserve discriminative spectral
information. Overall, for Salinas, the combination “MPCA + (RF, SVM, ExtraTrees or
KNN)” with $d = 10$ offers an excellent trade-off between accuracy and computation.

\subsubsection{FSDD}

\begin{table}[htp]
\centering
\caption{Results for the FSDD dataset.}
\begin{tabular}{lcccccc}
\hline
\textbf{d} 
& \textbf{MPCA} 
& \textbf{MKPCA} 
& \textbf{MONPP} 
& \textbf{MKONPP} 
& \textbf{MLLE} 
& \textbf{MLE} \\
\hline
\multicolumn{7}{c}{\textbf{RF - Accuracy (\%)}} \\
\hline
$d = 5$  & 81.25 & 76.05 & 59.35 & 44.55 & 81.75 & 85.65 \\
$d = 10$ & 90.65 & 85.95 & 61.45 & 54.60 & 86.25 & 87.95 \\
\hline
\multicolumn{7}{c}{\textbf{KNN - Accuracy (\%)}} \\
\hline
$d = 5$  & 83.10 & 76.70 & 58.00 & 41.95 & 79.80 & 83.70 \\
$d = 10$ & 90.60 & 87.90 & 59.40 & 55.20 & 82.90 & 83.75 \\
\hline
\multicolumn{7}{c}{\textbf{SVM - Accuracy (\%)}} \\
\hline
$d = 5$  & 88.90 & 82.70 & 55.25 & 43.05 & 84.60 & 87.35 \\
$d = 10$ & 95.95 & 92.90 & 57.33 & 56.55 & 87.80 & 89.30 \\
\hline
\multicolumn{7}{c}{\textbf{ExtraTrees - Accuracy (\%)}} \\
\hline
$d = 5$  & 84.30 & 78.80 & 61.00 & 45.90 & 84.05 & 87.00 \\
$d = 10$ & 92.30 & 89.15 & 62.65 & 57.90 & 87.70 & 87.80 \\
\hline
\multicolumn{7}{c}{\textbf{CPU-time (s)}} \\
\hline
$d = 5$  & 1085.58 & 10.08 & 137.28 & 21.47 & 7.57 & 6.09 \\
$d = 10$ & 1082.25 & 11.22 & 137.64 & 21.27 & 10.36 & 6.31 \\
\hline
\end{tabular}
\label{tab:fsdd}
\end{table}

\nd For the audio FSDD dataset (Table~\ref{tab:fsdd}), MPCA combined with SVM
achieves the best performance, reaching $95.95\%$ at $d = 10$. MLE and MKPCA also
perform competitively, especially for higher dimension, whereas MONPP and MKONPP
lag significantly behind. The CPU-times show the computational cost of MPCA
on this large tensor (around $10^3$ seconds). MKPCA and MLE are much faster.
When computational resources are limited, MKPCA or MLE with a strong classifier
(e.g., SVM or ExtraTrees) can be used as a good compromise between accuracy and efficiency.

\subsubsection{Gait}

\begin{table}[htp]
\centering
\caption{Results for the Gait dataset.}
\begin{tabular}{lcccccc}
\hline
\textbf{d} 
& \textbf{MPCA} 
& \textbf{MKPCA} 
& \textbf{MONPP} 
& \textbf{MKONPP} 
& \textbf{MLLE} 
& \textbf{MLE} \\
\hline
\multicolumn{7}{c}{\textbf{RF - Accuracy (\%)}} \\
\hline
$d = 5$  & 98.27 & 98.00 & 96.80 & 61.73 & 65.73 & 88.80 \\
$d = 10$ & 98.67 & 98.67 & 96.80 & 69.73 & 68.27 & 91.07 \\
\hline
\multicolumn{7}{c}{\textbf{KNN - Accuracy (\%)}} \\
\hline
$d = 5$  & 94.93 & 95.47 & 94.80 & 63.20 & 92.40 & 91.07 \\
$d = 10$ & 94.67 & 95.20 & 94.67 & 68.40 & 56.53 & 84.53 \\
\hline
\multicolumn{7}{c}{\textbf{SVM - Accuracy (\%)}} \\
\hline
$d = 5$  & 97.73 & 97.47 & 55.47 & 65.87 & 45.20 & 57.33 \\
$d = 10$ & 98.53 & 98.53 & 58.80 & 74.93 & 63.87 & 71.87 \\
\hline
\multicolumn{7}{c}{\textbf{ExtraTrees - Accuracy (\%)}} \\
\hline
$d = 5$  & 99.07 & 98.80 & 97.07 & 65.33 & 66.13 & 89.07 \\
$d = 10$ & 99.20 & 98.93 & 97.73 & 70.80 & 68.80 & 91.87 \\
\hline
\multicolumn{7}{c}{\textbf{CPU-time (s)}} \\
\hline
$d = 5$  & 1122.49 & 9.04 & 158.81 & 8.00 & 5.07 & 4.02 \\
$d = 10$ & 1081.84 & 5.91 & 118.99 & 9.29 & 4.84 & 4.63 \\
\hline
\end{tabular}
\label{tab:gait}
\end{table}

\nd For the Gait dataset (Table~\ref{tab:gait}), all PCA-based approaches
(MPCA, MKPCA, MONPP) achieve very high accuracy, especially when combined with
tree-based classifiers. The best performance is obtained by MPCA or MKPCA with
ExtraTrees at $d = 10$ (around $99\%$). This illustrates that  relatively low-dimensional
embeddings suffice to preserve most of the discriminative temporal–spatial gait
information. MLLE and MLE are less competitive for this dataset. In terms of
runtime, MPCA is the most expensive, while MKPCA and MLE are much cheaper.

\subsubsection{Washington DC Mall}

\begin{table}[htp]
\centering
\caption{Results for the Washington DC dataset.}
\begin{tabular}{lcccccc}
\hline
\textbf{d} 
& \textbf{MPCA} 
& \textbf{MKPCA} 
& \textbf{MONPP} 
& \textbf{MKONPP} 
& \textbf{MLLE} 
& \textbf{MLE} \\
\hline
\multicolumn{7}{c}{\textbf{RF - Accuracy (\%) }} \\
\hline
$d = 5$  & 95.40  & 95.25  & 89.00  & 77.85  & 94.25  & 95.55  \\
$d = 10$ & 96.05  & 95.95  & 92.30  & 80.50  & 94.50  & 95.50  \\
\hline
\multicolumn{7}{c}{\textbf{KNN - Accuracy (\%)}} \\
\hline
$d = 5$  & 93.30 & 94.30 & 82.70 & 76.95 & 93.05 & 95.45 \\
$d = 10$ & 92.90 & 94.70 & 87.85 & 80.30 & 93.05 & 95.25 \\
\hline
\multicolumn{7}{c}{\textbf{SVM - Accuracy (\%)}} \\
\hline
$d = 5$  & 95.20 & 95.10 & 82.25 & 78.70 & 93.55 & 95.00 \\
$d = 10$ & 95.85 & 96.40 & 88.25 & 82.45 & 93.45 & 95.45 \\
\hline
\multicolumn{7}{c}{\textbf{ExtraTrees - Accuracy (\%)}} \\
\hline
$d = 5$  & 95.70 & 95.55 & 89.90 & 77.85 & 94.45 & 96.25 \\
$d = 10$ & 96.80 & 96.25 & 93.15 & 81.50 & 94.95 & 96.05 \\
\hline
\multicolumn{7}{c}{\textbf{CPU-time (s)}} \\
\hline
$d = 5$  & 229.89 & 9.44 & 36.28 & 16.46 & 5.63 & 5.11 \\
$d = 10$ & 229.68 & 9.10 & 36.47 & 15.35 & 7.75 & 4.83 \\
\hline
\end{tabular}
\label{tab:washingtondc}
\end{table}

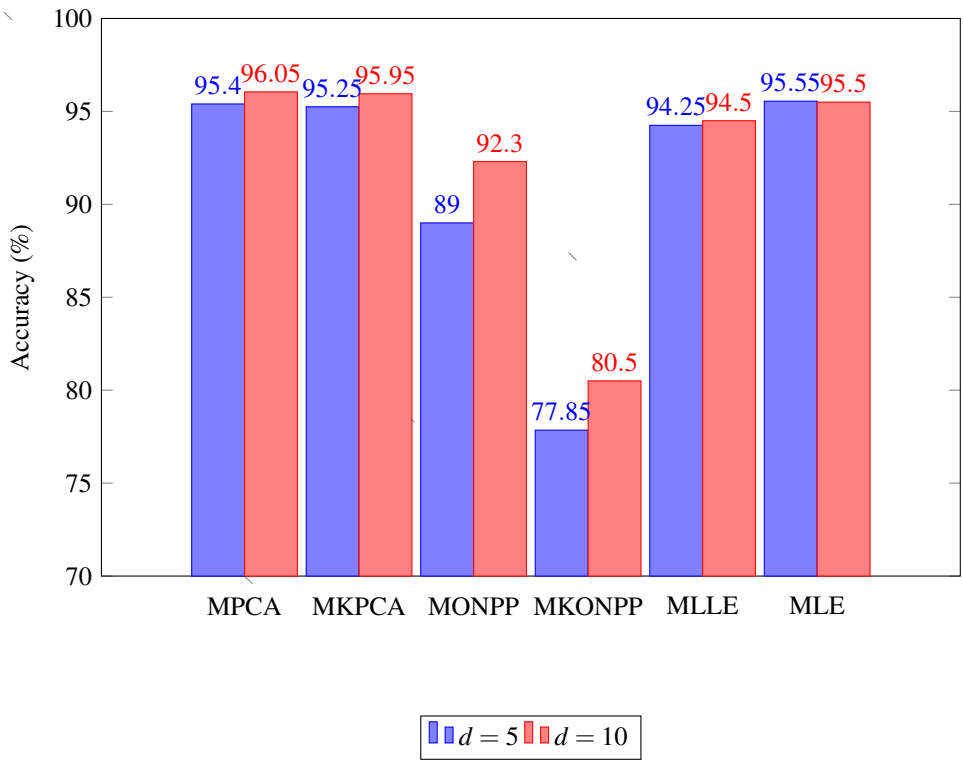
\begin{figure}[htp]
\centering
\begin{tikzpicture}
  \begin{axis}[
    ybar=0pt,
    bar width=20pt,
    width=13cm,
    height=9cm,
    enlarge x limits=0.25,
    ylabel={Accuracy (\%)},
    symbolic x coords={MPCA, MKPCA, MONPP, MKONPP, MLLE, MLE},
    xtick=data,
    xtick style={rotate=45, anchor=east},
    legend style={at={(0.5,-0.25)}, anchor=north, legend columns=2},
    ymin=70, ymax=100,
    nodes near coords,
    nodes near coords align={vertical},
  ]
    \addplot+[fill=blue!50] coordinates {
      (MPCA,95.40) (MKPCA,95.25) (MONPP,89.00) (MKONPP,77.85) (MLLE,94.25) (MLE,95.55)
    };
    \addplot+[fill=red!50] coordinates {
      (MPCA,96.05) (MKPCA,95.95) (MONPP,92.30) (MKONPP,80.50) (MLLE,94.50) (MLE,95.50)
    };
    \legend{$d=5$, $d=10$}
  \end{axis}
\end{tikzpicture}
\caption{Classification accuracies on the Washington-DC dataset for $d=5$ and $d=10$.}
\label{fig:washingtondc_bar}
\end{figure}

\medskip
  
\nd Table \ref{tab:washingtondc} reports the accuracy for various dimensionality-reduction methods (MPCA, MKPCA, MONPP, MKONPP, MLLE, MLE) applied to the Washington-DC dataset, for two reduced dimensions \(d = 5\) and \(d = 10\) and several classifiers (RF, KNN, SVM, ExtraTrees). 
Figure~\ref{fig:washingtondc_bar} illustrates the impact of the target dimension on accuracy for this dataset. Overall, the impact of increasing the target dimension from $d=5$ to $d=10$ is modest but consistently positive across the methods, especially for MPCA and MKPCA, confirming that relatively low multilinear dimensions suffice for accurate classification performance on this scene.

\subsubsection{MNIST}

\begin{table}[htp]
\centering
\caption{Results for the MNIST dataset.}
\begin{tabular}{lcccccc}
\hline
\textbf{d} 
& \textbf{MPCA} 
& \textbf{MKPCA} 
& \textbf{MONPP} 
& \textbf{MKONPP} 
& \textbf{MLLE} 
& \textbf{MLE} \\
\hline
\multicolumn{7}{c}{\textbf{RF - Accuracy (\%)}} \\
\hline
$d = 5$  & 81.50 & 79.90 & 82.45 & 77.55 & 80.40 & 86.30 \\
$d = 10$ & 87.35 & 84.15 & 88.30 & 83.10 & 87.30 & 89.10 \\
\hline
\multicolumn{7}{c}{\textbf{KNN - Accuracy (\%)}} \\
\hline
$d = 5$  & 83.80 & 79.65 & 83.55 & 75.35 & 77.20 & 85.10 \\
$d = 10$ & 90.20 & 87.10 & 89.85 & 82.50 & 85.90 & 89.05 \\
\hline
\multicolumn{7}{c}{\textbf{SVM - Accuracy (\%)}} \\
\hline
$d = 5$  & 87.00 & 84.30 & 86.35 & 78.70 & 81.40 & 87.75 \\
$d = 10$ & 92.35 & 89.65 & 91.95 & 86.95 & 88.60 & 90.50 \\
\hline
\multicolumn{7}{c}{\textbf{ExtraTrees - Accuracy (\%)}} \\
\hline
$d = 5$  & 84.45 & 81.15 & 83.90 & 77.95 & 81.75 & 87.10 \\
$d = 10$ & 90.10 & 85.80 & 89.45 & 84.15 & 87.15 & 89.65 \\
\hline
\multicolumn{7}{c}{\textbf{CPU-time (s)}} \\
\hline
$d = 5$  & 3.95 & 7.33 & 3.67 & 7.81 & 4.17 & 3.54 \\
$d = 10$ & 3.69 & 7.42 & 3.74 & 9.12 & 5.76 & 3.67 \\
\hline
\end{tabular}
\label{tab:mnist_results}
\end{table}

\nd For the MNIST dataset (Table~\ref{tab:mnist_results}), SVM consistently
delivers the best classification performance regardless of the dimensionality
reduction method. Among the latter methods, MPCA, MONPP, and MLE perform the best for
$d = 10$, with accuracies around $90\%$ or higher, while MKONPP and MKPCA yield slightly lower accuracy. The results show that projecting digit images down to $d = 10$
preserves enough discriminative information for competitive recognition accuracy, while substantially reducing the original image dimension.

\subsubsection{Effect of the target dimension on FSDD}

\begin{figure}[htp]
    \centering
    \includegraphics[width=0.6\textwidth]{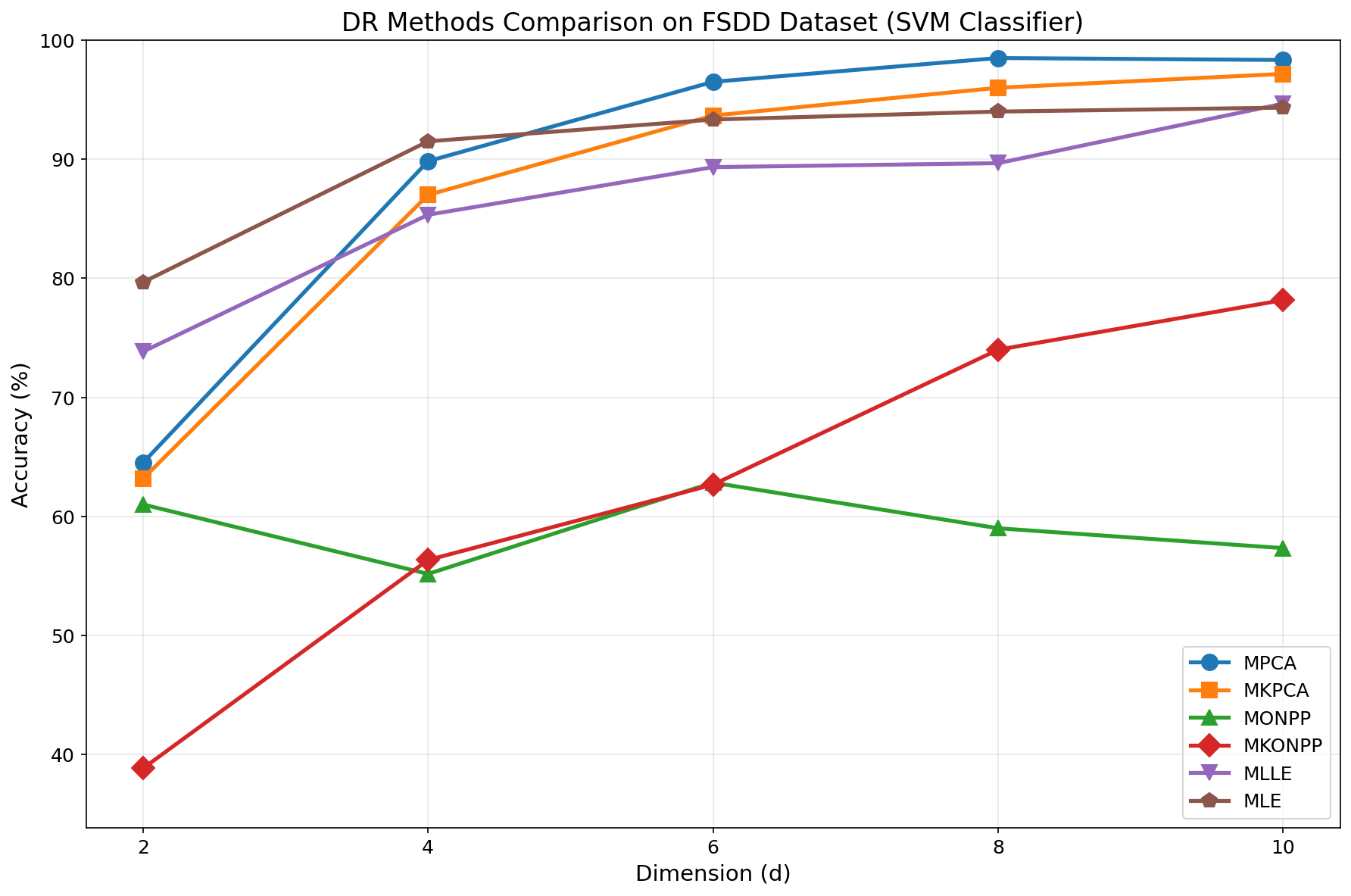}
    \caption{Classification accuracy comparison of multilinear dimensionality 
    reduction methods on the FSDD dataset with target dimensions $d = 2,4,6,8,10$.}
    \label{fig:dr_fsdd}
\end{figure}

\nd Figure~\ref{fig:dr_fsdd} displays the evolution of SVM classification
accuracy when used for the FSDD dataset as a function of the target dimension
$d \in \{2,4,6,8,10\}$. MPCA is the most effective dimensionality reduction method on this dataset, closely followed by MKPCA and MLE. MLLE is competitive for larger $d$-values, whereas NPP-based methods (MONPP and MKONPP) remain
less suitable for this particular audio classification task.

\nd Overall, the experiments show that modest multilinear dimensionality reduction
($d \leq 10$) is sufficient to reach high classification accuracies on a wide
range of multidimensional data types, and that the proposed tensor-based methods,
in particular MPCA and its kernel extension, offer a favourable balance between
discriminative power and computational efficiency.

\section{Conclusion}\label{sec:concl}

In this paper, we have extended several classical matrix-based dimensionality reduction techniques to the multidimensional setting by exploiting the algebraic structure of third-order tensors through the t-product formalism. This tensor framework preserves the inherent multiway correlations of the data, which are often destroyed when higher-order arrays are vectorized or flattened into matrices.

By recasting the usual trace-optimization criteria in the tensor domain, we obtained multidimensional counterparts of standard linear methods such as PCA and ONPP, together with nonlinear manifold learning techniques including LLE and Laplacian Eigenmaps. We also introduced kernelized tensor variants of PCA and ONPP, that allow us to capture nonlinear relationships while still operating directly on tensor-structured data. Altogether, the proposed methods offer a unified and flexible toolkit for processing high-dimensional multiway data in applications such as image and video analysis, remote sensing, signal processing, and multidimensional data visualization.

The numerical experiments on a diverse set of benchmark datasets show these tensor-based methods to be competitive, and often superior, with regard to classification performance while reducing the dimensionality of the data by one to two orders of magnitude. In particular, MPCA and its kernel extension generally provide very favorable trade-offs between accuracy and computational cost, whereas manifold-based methods (MLLE and MLE) may further improve classification performance in some settings at the price of higher complexity. Across all datasets, the reduced representations were evaluated using standard classifiers such as SVM, KNN, Random Forest, and Extra Trees, highlighting the versatility of the proposed tensor approaches as generic preprocessing tools for supervised learning on multidimensional data.

\section*{Acknowledgment}
This Project was funded by the Deanship of Scientific Research (DSR) at King Abdulaziz University, Jeddah, Saudi Arabia under grant no. IPP: 1460-665-2025. The authors acknowledge with thanks DSR for technical and financial support.

\bibliographystyle{plain}  
\bibliography{bibliography}    

\end{document}